\documentclass[a4paper,12pt,reqno]{amsart}
\usepackage{amsmath,graphicx,graphics,amssymb}
\newtheorem{theo}{Theorem}[section]
\newtheorem{lem}[theo]{Lemma}

\newtheorem{prop}[theo]{Proposition}

\numberwithin{equation}{section}

\newcommand{\M}{\operatorname{M}}
\newcommand{\w}{\operatorname{w}}

\newcommand{\per}{\operatorname{per}}
\newcommand{\sgn}{\operatorname{sgn}}
\newcommand{\Z}{\mathbb{Z}}
\newcommand{\la}{\lambda}
\newcommand{\ve}{\varepsilon}
\newcommand{\al}{\alpha}
\newcommand{\be}{\beta}
\newcommand{\epf}{\hfill{$\square$}\medskip}

\newmuskip\pFqskip
\mathchardef\pFcomma=\mathcode`, 

\begin{document}

\title{A new solution for the two dimensional dimer problem}

\author{Mihai Ciucu}
\address{Department of Mathematics, Indiana University, Bloomington, Indiana 47405}

\thanks{Research supported in part by NSF grant DMS-1501052 and Simons Foundation Collaboration Grant 710477}

\begin{abstract} The classical 1961 solution to the problem of determining the number of perfect matchings (or dimer coverings) of a rectangular grid graph --- due independently to Kasteleyn and to Temperley and Fisher --- consists of changing the sign of some of the entries in the adjacency matrix so that the Pfaffian of the new matrix gives the number of perfect matchings, and then evaluating this Pfaffian. Another classical method is to use the Lindstr\"om-Gessel-Viennot theorem on non-intersecting lattice paths to express the number of perfect matchings as a determinant, and then evaluate this determinant. In this paper we present a new method for solving the two dimensional dimer problem, which relies on the Cauchy-Binet theorem. It only involves facts that were known in the mid 1930's when the dimer problem was phrased, so it could have been discovered while the dimer problem was still open.

We provide explicit product formulas for both the square and the hexagonal lattice. One advantage of our formula for the square lattice compared to the original formula of Kasteleyn, Temperley and Fisher is that ours has a linear number of factors, while the number of factors in the former is quadratic. Our result for the hexagonal lattice yields a formula for the number of periodic stepped surfaces that fit in an infinite tube of given cross-section, which can be regarded as a counterpart of MacMahon's boxed plane partition theorem. 
  
\end{abstract}

\maketitle

%
%

\section{Introduction}



In 1961 Temperley and Fisher \cite{TF,Fisher-dimer}, and Kasteleyn \cite{Kast}, independently determined the number of perfect matchings of a rectangular grid graph, thus solving the dimer problem posed in 1937 by Fowler and Rushbrooke \cite{FR} in the limiting case when the dimers completely fill the lattice (the so called close-packed or high density limit); we refer to the formula they discovered --- included in this paper as equation \eqref{eff} --- as the TFK formula. Both solutions consist of changing the sign of some of the entries in the adjacency matrix of the grid graph so that the Pfaffian of the new matrix gives the number of perfect matchings, and then evaluating this Pfaffian. The resulting formula was then analyzed asymptotically to calculate the so called free energy per site --- the limit of the ratio between the logarithm of the number of perfect matchings and the number of vertices in the rectangular graph.

Even though this was not realized until later, MacMahon \cite{MacM} solved an instance of the dimer problem on the hexagonal lattice (in the equivalent language of plane partitions; see \cite{DT}), by providing a simple product formula for the number of perfect matchings of honeycomb graphs (centrally symmetric hexagonal portions of the hexagonal lattice).

A modern way of proving MacMahon's result is to use the Lindstr\"om-Gessel-Viennot theorem (see \cite{L,GV,Ste}) on non-intersecting lattice paths to express the number of perfect matchings as a determinant, and then evaluate this determinant.

The Pfaffian method and the non-intersecting lattice paths method are two classical tools for determining the number of perfect matchings of a planar graph. In this paper we present a new method, based on the Cauchy-Binet theorem, and use it to give a new solution to the dimer model on the hexagonal and on the square lattice. This solution only involves facts that were known in the mid 1930's when the dimer problem was phrased, so it could have been discovered while the dimer problem was still open.

The particular form of the involved graphs has cylindrical boundary conditions, and the exact product formulas we obtain seem not to have appeared previously in the literature.

This paper is organized as follows. In Section 2 we present the new method, and state the general result as Theorem \ref{tba}. In Section 3 we apply the general result to honeycomb graphs embedded in a cylinder, and find an explicit product formula for the number of their perfect matchings (see Theorem \ref{tca}). In Section 4 we connect these results to the enumeration of cylindric partitions of a certain double staircase shape. Theorem \ref{tda} gives a formula for the number of periodic stepped surfaces that fit in an infinite tube of given cross-section. It can be regarded as a counterpart of MacMahon's boxed plane parition theorem \cite{MacM} (see Remark 5).

In Section 5 we turn to applications involving the square lattice. We show how Theorem \ref{tba} can be used to find an explicit product formula for the number of perfect matchings of square cylinder graphs of even girth (see Theorem \ref{tea}).

Square cylinder graphs of odd girth are dealt with in Section 7. We use there an extension of our factorization theorem from \cite{FT} (which we present in Section 6; see Theorem \ref{tga}) to obtain a common formula for the number of perfect matchings of square cylinders of even or odd girth. Section 8 contains some concluding remarks.

\begin{figure}[t]
\vskip0.1in
\centerline{
\hfill
{\includegraphics[width=0.65\textwidth]{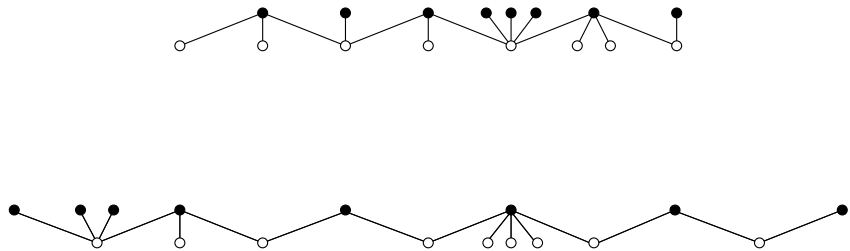}}
\hfill
}
\caption{\label{fba} Two strands.}
\end{figure}

\begin{figure}[t]
\vskip0.2in
\centerline{
\hfill
{\includegraphics[width=0.65\textwidth]{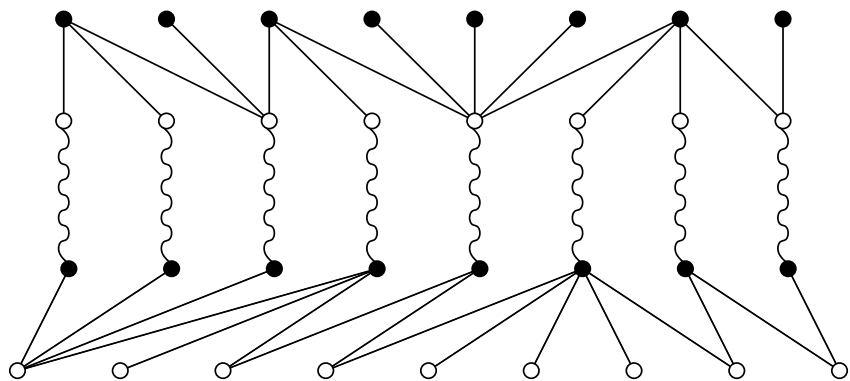}}
\hfill
}
\caption{\label{fbb} Connecting the two strands in Figure \ref{fba}.}
\end{figure}

\section{The new method}

A {\it strand} is a path with pending edges (see Figure \ref{fba} for an illustrative example); the edges in a strand are allowed to carry arbitrary weights (in particular, these weights are allowed to equal zero). Thus, strands are weighted bipartite graphs. We always draw strands so that their vertices are lined up along two horizontal lines, with white vertices on the bottom and black vertices on top.

Our new method for enumerating dimer coverings applies to a special kind of planar graphs built from strands, which we call {\it fabric graphs}. There are two versions of them. Both consist of horizontal strands joined together by vertical edges.

A {\it rectangular fabric graph} is a weighted graph obtained from horizontal strands $S_1,\dotsc,S_m$ as follows. Draw the strands in the plane successively one above the other, with $S_1$ on the bottom. Denote the number of white (bottom) vertices in $S_i$ by $k_i$, and the number of black (top) vertices by $l_i$, $i=1,\dotsc,m$. Assume $l_i=k_{i+1}$, for $i=1,\dotsc,m-1$, and join the top vertices of the strand $S_i$ to the bottom vertices of the strand $S_{i+1}$ by vertical edges
of weight 1
(see Figure~\ref{fbb}; for easy distinguishing, vertical edges between the strands are indicated as wavy lines). The resulting graph is called a rectangular fabric graph.

The second family consists of {\it cylindrical fabric graphs}. Start as above with strands $S_1,\dotsc,S_m$, but assume now that $l_i=k_{i+1}$, for $i=1,\dotsc,m$, where $k_{m+1}:=k_1$.
Then, in addition to joining together strands $S_i$ and $S_{i+1}$ by vertical edges for $i=1,\dotsc,m-1$, as in the previous paragraph, join together also the top vertices of the strand $S_m$ with the bottom vertices of the strand $S_1$, by vertical edges (which can be thought of as curving down behind a horizontal cylinder). Furthermore, more generally than for the case of their rectangular counterparts, weight the vertical edges of cylindrical fabric graphs as follows: weight the vertical edges connecting strand $S_i$ to strand $S_{i+1}$ by $x_i$, $i=1,\dotsc,m$ (where $S_{m+1}:=S_1$).

The {\it bi-adjacency} matrix of strand $S_i$ is the matrix $A_i$ defined as follows. The rows of $A_i$ are indexed by the bottom vertices of $S_i$ (listed from left to right), and the columns of $A_i$ are indexed by the top vertices of $S_i$ (also listed from left to right). Then if $u$ is a bottom vertex of $S_i$ and $v$ is a top vertex of $S_i$, define the $(u,v)$-entry of the bi-adjacency matrix $A_i$ to be the weight of the edge $\{u,v\}$, if $u$ and $v$ are connected by an edge in $S_i$, and zero otherwise.

In general, given a weighted graph $G$, we denote by $\M(G)$ the sum of the weights of all perfect matchings\footnote{ A perfect matching  $\mu$ of a graph $G$ is a collection of edges of $G$ with the property that every vertex of $G$ is incident to precisely one edge in $\mu$. The weight of $\mu$ is defined to be the product of the weights of the edges in $\mu$. If $G$ is not weighted (equivalently, all edges of $G$ have weight 1), then $\M(G)$ is simply the number of perfect matchings of $G$.} of $G$. However, if $G$ is a cylindrical fabric graph, we often write $\M(G;x_1,\dotsc,x_m)$ (or $\M(G;x)$, if $x_1=\cdots=x_m=x$) for the sum of the weights of its perfect matchings, in order to spell out the weights of the vertical edges.

As a consequence of their definition, both rectangular fabric graphs and cylindrical fabric graphs are bipartite. A necessary condition for a bipartite graph to admit a perfect matching is to be {\it balanced} (i.e., to have the same number of vertices in the two bipartition classes). The definition of cylindrical fabric graphs implies that they are balanced. A rectangular fabric graph has $k_1+\cdots+k_m$ white and $l_1+\cdots+l_m$ black vertices, and since $l_i=k_{i+1}$, $i=1,\dotsc,m-1$, it is balanced if and only if $k_1=l_m$.

Therefore, without loss of generality, we may assume that for any fabric graph (be it rectangular or cylindrical) with $m$ strands, there exist non-negative integers $l_1,\dotsc,l_m$ so that the $i$th strand has $l_{i-1}$ bottom vertices and $l_i$ top vertices, $i=1,\dotsc,m$ (with $l_0=l_m$).

\begin{theo}
\label{tba}  
$(${\rm a}$)$. If $G$ is a balanced rectangular fabric graph with $m$ strands, then
\begin{equation}
\M(G)=\det(A_1A_2\cdots A_m).
\label{eba}
\end{equation}

\parindent0pt
$(${\rm b}$)$. If $G$ is a cylindrical fabric graph with $m$ strands, strand $S_i$ having $l_{i-1}$ bottom vertices and $l_i$ top vertices, $i=1,\dotsc,m$, we have
\begin{equation}
\M(G;x_1,\dotsc,x_m)=x_1^{l_1-l_m}\cdots x_{m-1}^{l_{m-1}-l_{m}}\det((x_1\cdots x_m)I+A_1A_2\cdots A_m),
\label{ebb}
\end{equation}
where $I$ is the identity matrix of order $l_m$.
\end{theo}

\parindent0pt
Our proof of Theorem \ref{tba} is based on three preliminary lemmas. Recall that the permanent $\per A$ of an $n\times n$ matrix $A=(a_{ij})$ is defined to be
\begin{equation}
\per A=\sum_{\pi\in {\mathcal S}_n} a_{1,\pi(1)} a_{2,\pi(2)}\cdots a_{n,\pi(n)},
\label{eper}
\end{equation}
where ${\mathcal S}_n$ is the set of permutations of order $n$. The following lemma is well known.

\begin{lem}
\label{tbb}
Let $G$ be a weighted bipartite graph with the same number of white and black vertices. Let $A$ be the bi-adjacency matrix of $G$ $($i.e., the rows of $A$ are indexed by the white vertices, the columns by the black vertices, and the $(u,v)$ entry is the weight of the edge $\{u,v\}$, or $0$ if there is no edge between $u$ and $v$$)$. Then 
\begin{equation}
\M(G)=\per A.
\label{etbb}
\end{equation}

\end{lem}

\begin{proof} Each non-zero term in the expansion of $\per A$ (see the right hand side of \eqref{eper}) corresponds to a collection of $n$ edges of $G$ that share no white vertex (because no two $a_{i,\pi(i)}$'s are in the same row) and no black vertex (because no two $a_{i,\pi(i)}$'s are in the same column), i.e. to a perfect matching of $G$.
\end{proof}

\begin{lem}
\label{tbc}
Let $S$ be a strand $($see definition at the beginning of this section$)$, and let $A$ be the bi-adjacency matrix of $S$, with the rows of $A$ indexed by the bottom vertices of $S$ listed from left to right, and the columns of $A$ indexed by the top vertices of $S$ listed from left to right. Then for any square submatrix $B$ of $A$ we have
\begin{equation}
\per B=\det B.
\label{etbc}
\end{equation}

\end{lem}

\begin{proof} One readily sees that for the indicated ordering of the vertices of the strand, if the entries of the bi-adjacency matrix $A$ are thought of as residing at the centers of the unit squares of an $n\times n$ chessboard, the support of $A$ is contained in a zigzag strip --- the path of a rook that is allowed to move only down or to the right. This implies that in the expansion 
\begin{equation}
\det B=\sum_{\sigma\in{\mathcal S}_k}(-1)^{\sgn(\sigma)}b_{1,\sigma(1)}\cdots b_{k,\sigma(k)}
\end{equation}
of the determinant of any $k\times k$ submatrix $B=(b_{i,j})_{1\leq i,j\leq k}$ of $A$, all the non-zero terms correspond to permutations $\sigma$ with no inversions.
\end{proof}

\begin{lem}
\label{tbd}
Let $A$ be an $n\times n$ matrix. Then\footnote{ Following customary notation, we set $[n]:=\{1,\dotsc,n\}$. As usual, $A_I^J$ denotes the submatrix of $A$ obtained by choosing its elements in rows with indices in the set $I$ and columns with indices in the set $J$.}
\begin{equation}
\sum_{J\subset[n]}x^{n-|J|}\det A_J^J=\det (xI_n+A),
\label{etbd}
\end{equation}
where $I_n$ is the identity matrix of order $n$.
\end{lem}

\begin{proof} Regard the $j$th column of the matrix $xI_n+A$ as the sum of the $j$th column of $xI_n$ and the $j$th column of $A$, for $j=1,\dotsc,n$. Using the linearity of the determinant in columns, $\det (xI_n+A)$ becomes a sum of $2^n$ determinants. Each term in this sum is the determinant of a matrix obtained from $A$ by picking some subset $J\subset[n]$, keeping the entries of $A$ in the rows and columns indexed by $J$, and replacing them by the entries of $xI_n$ in the rows and columns with indices outside $J$. But this determinant is clearly equal to $x^{n-|J|}\det A_J^J$.
\end{proof}

\parindent0pt
\medskip
{\it Proof of Theorem $\ref{tba}$.} $(${\rm a}$)$. Perfect matchings of a rectangular fabric graph can be thought of as being built in two stages: (1) specify which vertical edges participate in each level --- this is equivalent to specifying subsets $J_i\subset[l_i]$, for $i=1,\dotsc,m-1$; and (2) then choosing the perfect matching {\it internally} within each horizontal strand, independently.

\parindent12pt
Denote by $S_i(I,J)$ the subgraph of the strand $S_i$ obtained by deleting bottom vertices with labels\footnote{ The bottom vertices in a strand are labeled from left to right by consecutive integers starting with 1; similarly the top vertices.} in $I\subset[k_i]$ and top vertices with labels in $J\subset[l_i]$. Viewing perfect matchings as described in the previous paragraph we obtain

\begin{align}
\M(G)&=\sum_{J_1\subset[l_1],\dotsc,J_{m-1}\subset[l_{m-1}]}\M(S_1(\emptyset,J_1))\M(S_2(J_1,J_2))
\cdots\M(S_2(J_{m-2},J_{m-1}))\M(S_2(J_{m-1},\emptyset))
\nonumber
\\[10pt]
&=
\sum_{J_1\subset[l_1],\dotsc,J_{m-1}\subset[l_{m-1}]}\per (A_1)_{[l_m]}^{[l_1]\setminus J_1}
\per (A_2)_{[l_1]\setminus J_1}^{[l_2]\setminus J_2}\cdots
\per (A_{m-1})_{[l_{m-2}]\setminus J_{m-2}}^{[l_{m-1}]\setminus J_{m-1}}
\per (A_{m})_{[l_{m-1}]\setminus J_{m-1}}^{[l_m]},
\nonumber
\\[10pt]
&=
\sum_{J_1\subset[l_1],\dotsc,J_{m-1}\subset[l_{m-1}]}\det (A_1)_{[l_m]}^{[l_1]\setminus J_1}
\det (A_2)_{[l_1]\setminus J_1}^{[l_2]\setminus J_2}\cdots
\det (A_{m-1})_{[l_{m-2}]\setminus J_{m-2}}^{[l_{m-1}]\setminus J_{m-1}}
\det (A_{m})_{[l_{m-1}]\setminus J_{m-1}}^{[l_m]},
\label{ebc}
\end{align}  
where at the second and third equalities we used Lemmas \ref{tbb} and \ref{tbc}, respectively, and in the last two summations all the involved matrices are required to be square (i.e., $|J_1|=l_1-l_m$, $|J_{i}|-|J_{i-1}|=l_i-l_{i-1}$, $i=2,\dotsc,m-1$, $|J_{m-1}|=l_{m-1}-l_m$, which in turn is equivalent to $|J_i|=l_i-l_m$, $i=1,\dotsc,m-1$).

Using the Cauchy-Binet theorem (see e.g. \S4.6, pp. 208--214 in \cite{BW})  the right hand side of \eqref{ebc} can be written as

\begin{align}
  &
\sum_{J_2\subset[l_2],\dotsc,J_{m-1}\subset[l_{m-1}]\atop |J_2|=l_2-l_m,\dotsc,|J_{m-1}|=l_{m-1}-l_m}
\det (A_3)_{[l_2]\setminus J_2}^{[l_3]\setminus J_3}\cdots
\det (A_{m})_{[l_{m-1}]\setminus J_{m-1}}^{[l_m]}
\sum_{J_1\subset[l_1]\atop |J_1|=l_1-l_m}\det (A_1)_{[l_m]}^{[l_1]\setminus J_1}\det (A_2)_{[l_1]\setminus J_1}^{[l_2]\setminus J_2}
\nonumber
\\[10pt]
&\ \
=
\sum_{J_2\subset[l_2],\dotsc,J_{m-1}\subset[l_{m-1}]\atop |J_2|=l_2-l_m,\dotsc,|J_{m-1}|=l_{m-1}-l_m}
\det (A_1A_2)_{[l_m]}^{[l_2]\setminus J_2}
\det (A_3)_{[l_2]\setminus J_2}^{[l_3]\setminus J_3}\cdots
\det (A_{m})_{[l_{m-1}]\setminus J_{m-1}}^{[l_m]}.
\label{ebd}
\end{align}
\vskip0.1in

\parindent0pt
Repeating the argument that gave \eqref{ebd} $m-2$ more times yields equation \eqref{eba}. 

\parindent12pt
\medskip
$(${\rm b}$)$. If $G$ is a cylindrical fabric graph, the same reasoning that gave \eqref{ebc}
leads to
\begin{align}
&
\M(G;x_1,\dotsc,x_m)=
\nonumber
\\[10pt]
&\ \ 
\sum_{J_1\subset[l_1],\dotsc,J_{m}\subset[l_{m}]\atop |J_i|-|J_{i-1}|=l_i-l_{i-1},\ i=1,\dotsc,m}
x_1^{|J_1|}\cdots x_m^{|J_m|}
\det (A_1)_{[l_m]\setminus J_m}^{[l_1]\setminus J_1}
\det (A_2)_{[l_1]\setminus J_1}^{[l_2]\setminus J_2}\cdots
\det (A_{m})_{[l_{m-1}]\setminus J_{m-1}}^{[l_m]\setminus J_m}.
\label{ebe}
\end{align}  

Using repeatedly the Cauchy-Binet theorem as in the proof of part (a), and the fact that the conditions on the second row under the summation in \eqref{ebe} are equivalent to $|J_i|=|J_m|+l_i-l_m$, we obtain from \eqref{ebe} that

\begin{align}
\M(G;x_1,\dotsc,x_m)&=
x_1^{l_1-l_m}\cdots x_{m-1}^{l_{m-1}-l_m}
\sum_{J_{m}\subset[l_{m}]}
(x_1\cdots x_m)^{|J_m|}
\det (A_1A_2\cdots A_m)_{[l_m]\setminus J_m}^{[l_m]\setminus J_m}
\nonumber
\\[10pt]
&=x_1^{l_1-l_m}\cdots x_{m-1}^{l_{m-1}-l_m}
\sum_{J\subset[l_{m}]}
(x_1\cdots x_m)^{l_m-|J|}
\det (A_1A_2\cdots A_m)_{J}^{J}.
\label{ebj}
\end{align}  
However, by Lemma \ref{tbd}, the right hand side above is the same as the right hand side of equation~\eqref{ebb}. \hfill $\square$


\section{Honeycomb cylinder graphs}
  
\begin{figure}[t]
\vskip0.1in
\centerline{
\hfill
{\includegraphics[width=0.40\textwidth]{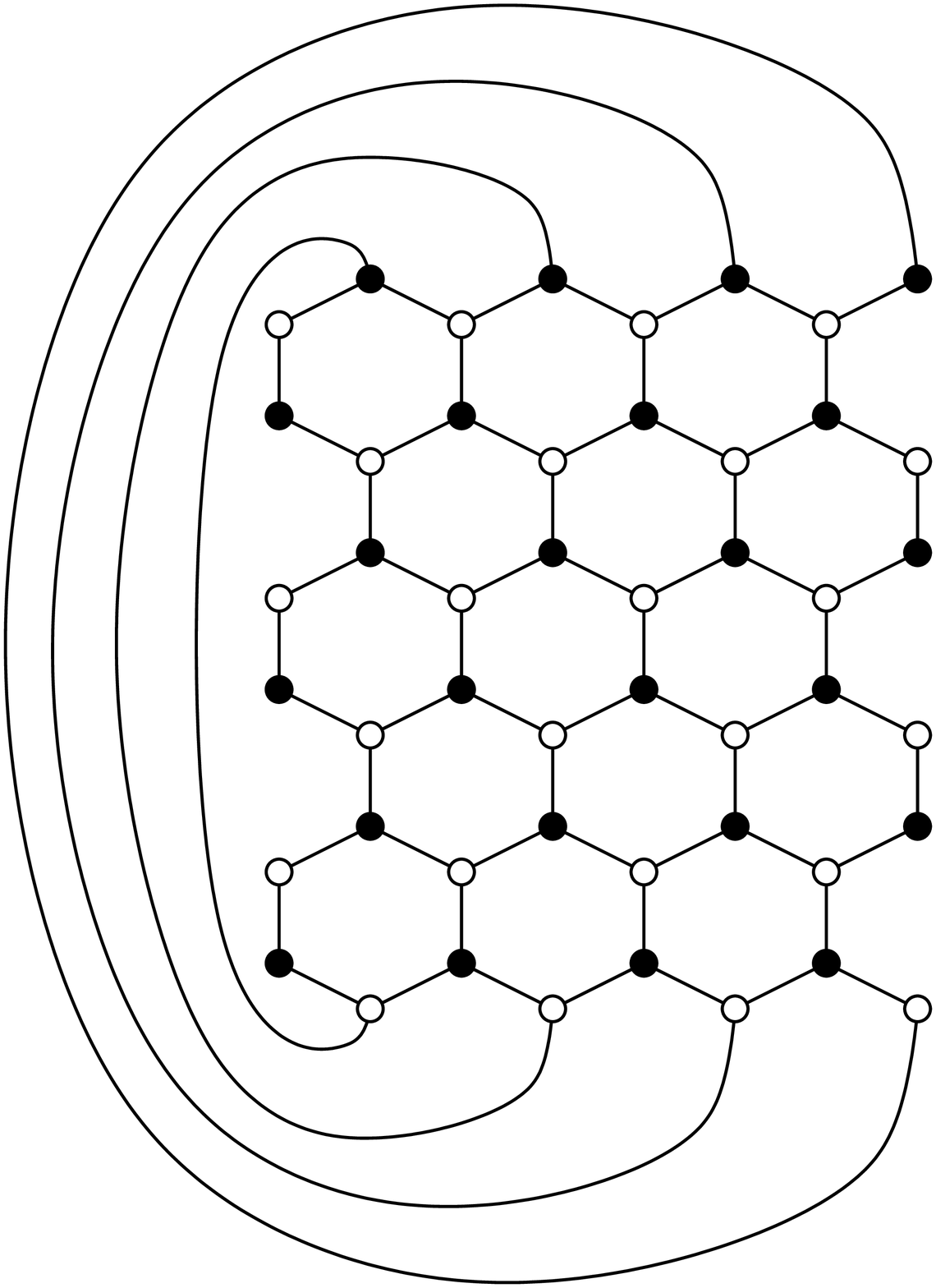}}
\hfill
{\includegraphics[width=0.43\textwidth]{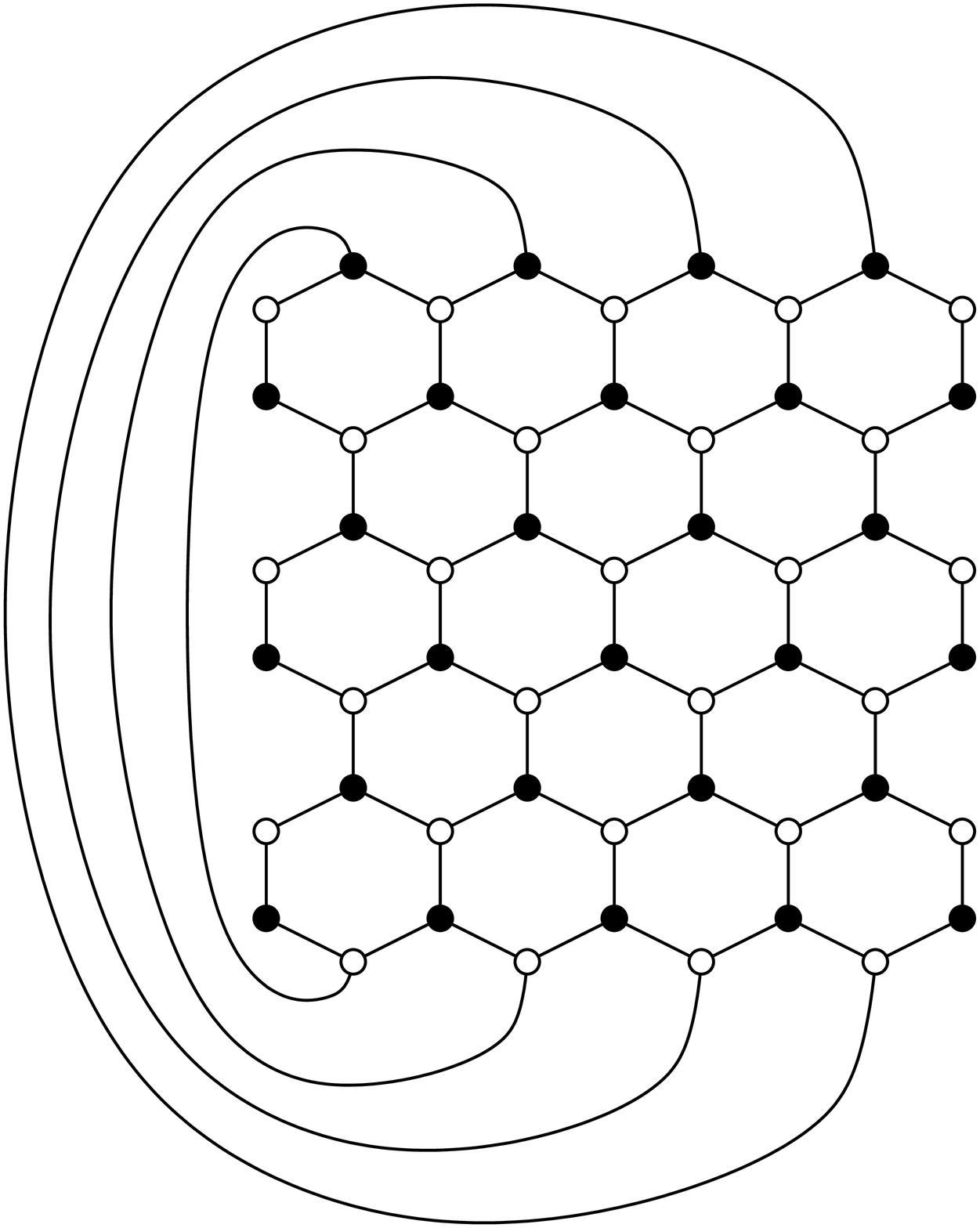}}
\hfill
}
\caption{\label{fca} The graph $H_{m,n}$ for $m=6$, $n=8$ (left) and $m=6$, $n=9$ (right).}
\end{figure}

In this section we consider the special case when the strands are paths of the same length. As we will see, in this case the determinant in Theorem \ref{tba} can be evaluated explicitly.

Define the {\it honeycomb cylinder graph} $H_{m,n}$ to be the cylindrical fabric graph with $m$ strands, each of which is a path with $n$ vertices (see Figure \ref{fca}). Note that the bottom and top strands of $H_{m,n}$ fit together ``seamlessly'' only when $m$ is even, in which case $H_{m,n}$ has $m/2$-fold rotational symmetry with respect to rotation around the cylinder in which it is naturally embedded. We will assume that $m$ is even throughout this section.

The main result of this section is the following.

\begin{theo}
\label{tca}
For $m$ even, the number of perfect matchings of the cylindrical honeycomb graph $H_{m,n}$ is given by
\begin{align}
  \M(H_{m,n};x_1,\dotsc,x_m)=
\begin{cases}
\hskip0.95in
\prod_{k=1}^{\lfloor n/2\rfloor}\left[x_1\cdots x_m+\left(2+2\cos{\frac{2k\pi}{n+1}}\right)^{m/2}\right], & n\ \text{\rm even},\\[10pt]
x_1x_3\cdots x_{m-1}
\prod_{k=1}^{\lfloor n/2\rfloor}\left[x_1\cdots x_m+\left(2+2\cos{\frac{2k\pi}{n+1}}\right)^{m/2}\right],
&n\ \text{\rm odd}.
\end{cases}
\label{eca}
\end{align}

\end{theo}

\parindent0pt
{\it Remark $1$.} The floor can be avoided at the upper limit of the product by writing
\begin{align}
\prod_{k=1}^{\lfloor n/2\rfloor}\left[x_1\cdots x_m+\left(2+2\cos{\frac{2k\pi}{n+1}}\right)^{m/2}\right]
=
\sqrt{\prod_{k=1}^{n}\left[x_1\cdots x_m+\left(2+2\cos{\frac{2k\pi}{n+1}}\right)^{m/2}\right]}.
\end{align}

In our proof we will employ the following lemma.

\parindent12pt

\begin{lem}
\label{tcb}
$(${\rm a}$)$. The eigenvalues of the $n\times n$ matrix
\begin{equation}
\left[\begin{matrix}
0 & 1 & 0 & \cdot & \cdot & \cdot & 0 & 0 & 0\\
1 & 0 & 1 & \cdot & \cdot & \cdot & 0 & 0 & 0\\
0 & 1 & 0 & \cdot & \cdot & \cdot & 0 & 0 & 0\\
\cdot & \cdot & \cdot & \cdot & \cdot & \cdot & \cdot & \cdot & \cdot\\
\cdot & \cdot & \cdot & \cdot & \cdot & \cdot & \cdot & \cdot & \cdot\\
\cdot & \cdot & \cdot & \cdot & \cdot & \cdot & \cdot & \cdot & \cdot\\
0 & 0 & 0 & \cdot & \cdot & \cdot & 0 & 1 & 0\\
0 & 0 & 0 & \cdot & \cdot & \cdot & 1 & 0 & 1\\
0 & 0 & 0 & \cdot & \cdot & \cdot & 0 & 1 & 0\\
\end{matrix}\right]
\label{ecb}
\end{equation}
are
\begin{equation}
2\cos\frac{k\pi}{n+1},\ \ \ k=1,\dotsc,n.
\label{ecc}
\end{equation}

\parindent0pt
$(${\rm b}$)$. The eigenvalues of the $n\times n$ matrix
\begin{equation}
\left[\begin{matrix}
0 & 1 & 0 & \cdot & \cdot & \cdot & 0 & 0 & 0\\
1 & 0 & 1 & \cdot & \cdot & \cdot & 0 & 0 & 0\\
0 & 1 & 0 & \cdot & \cdot & \cdot & 0 & 0 & 0\\
\cdot & \cdot & \cdot & \cdot & \cdot & \cdot & \cdot & \cdot & \cdot\\
\cdot & \cdot & \cdot & \cdot & \cdot & \cdot & \cdot & \cdot & \cdot\\
\cdot & \cdot & \cdot & \cdot & \cdot & \cdot & \cdot & \cdot & \cdot\\
0 & 0 & 0 & \cdot & \cdot & \cdot & 0 & 1 & 0\\
0 & 0 & 0 & \cdot & \cdot & \cdot & 1 & 0 & 1\\
0 & 0 & 0 & \cdot & \cdot & \cdot & 0 & 1 & -1\\
\end{matrix}\right]
\label{ecd}
\end{equation}
are
\begin{equation}
2\cos\frac{2k\pi}{2n+1},\ \ \ k=1,\dotsc,n.
\label{ece}
\end{equation}

\end{lem}

\parindent12pt

\begin{proof} Part (a) is a classical result. See for instance \cite[Section\ 2.6,\ \#7]{CDS}.

Denote by $q_n(\lambda)$ the characteristic polynomial of the matrix \eqref{ecd}. Regarding the last column of the defining determinant for $q_n(\lambda)$ as a sum of two columns and using the linearity of the determinant, one readily obtains that
\begin{equation}
q_n(\lambda)=p_n(\lambda)+p_{n-1}(\lambda),
\label{ecf}
\end{equation}  
where $p_n(\lambda)$ is the characteristic polynomial of the matrix \eqref{ecb}. However, the latter is just the characteristic polynomial of a path of length $n$. By \cite[Section\ 2.6,\ \#7]{CDS}, we have $p_n(\lambda)=U_n(\lambda/2)$, where $U_n(\lambda)$ is the Chebyshev polynomial of the second kind, given by
\begin{equation}
U_n(x)=\frac{\sin[(n+1)\arccos x]}{\sin(\arccos x)}.
\label{ecg}
\end{equation}  
The zeros of $q_n(\lambda)$ are then the zeros of $U_n(\lambda/2)+U_{n-1}(\lambda/2)$. We have
\begin{align}
U_n(x)+U_{n-1}(x)
&=\frac{\sin[(n+1)\arccos x]}{\sin(\arccos x)}+\frac{\sin[n\arccos x]}{\sin(\arccos x)}
\nonumber
\\[10pt]
&=\frac
{2\sin\left[\left(n+\frac12\right)\arccos x \right]\cos\left(\frac12\arccos x\right)}
{\sin(\arccos x)}.
\label{eci}
\end{align}  
Thus the zeros of $U_n(x)+U_{n-1}(x)$ are $\cos\frac{2k\pi}{2n+1}$, $k=1,\dotsc,n$, and those of
$q_n(\lambda)=U_n(\lambda/2)+U_{n-1}(\lambda/2)$ are $2\cos\frac{2k\pi}{2n+1}$, $k=1,\dotsc,n$. This proves part (b).
\end{proof}

\medskip
\parindent0pt
{\it Proof of Theorem $\ref{tca}$.} The details of the proof depend on the parity of $n$. We treat first the case when $n$ is even. Then the cylindrical hexagonal graph $H_{m,n}$ looks like in the picture on the left in Figure \ref{fca}. Let $A_n$ be the $n\times n$ matrix
\begin{equation}
A_n=
\left[\begin{matrix}
1 & 1 & 0 & \cdot & \cdot & \cdot & 0 & 0 & 0\\
0 & 1 & 1 & \cdot & \cdot & \cdot & 0 & 0 & 0\\
0 & 0 & 1 & \cdot & \cdot & \cdot & 0 & 0 & 0\\
\cdot & \cdot & \cdot & \cdot & \cdot & \cdot & \cdot & \cdot & \cdot\\
\cdot & \cdot & \cdot & \cdot & \cdot & \cdot & \cdot & \cdot & \cdot\\
\cdot & \cdot & \cdot & \cdot & \cdot & \cdot & \cdot & \cdot & \cdot\\
0 & 0 & 0 & \cdot & \cdot & \cdot & 1 & 1 & 0\\
0 & 0 & 0 & \cdot & \cdot & \cdot & 0 & 1 & 1\\
0 & 0 & 0 & \cdot & \cdot & \cdot & 0 & 0 & 1\\
\end{matrix}\right].
\label{ecj}
\end{equation}
Then if we denote the strands of $H_{m,n}$ by $S_1,\dotsc,S_m$ starting from the bottom, the bi-adjacency matrix of strand $S_i$ is $A_{n/2}$ for odd $i$, and its transpose $A_{n/2}^T$ for even $i$. Therefore, as each strand has $n/2$ black vertices, by Theorem \ref{tba}(b) we obtain
\begin{equation}
\M(H_{m,n};x_1,\dotsc,x_m)=\det\left(x_1\cdots x_m I_{n/2} + (A_{n/2}A_{n/2}^T)^{m/2}\right).
\label{eck}
\end{equation}  
We have
\begin{equation}
A_sA_s^T=
\left[\begin{matrix}
2 & 1 & 0 & \cdot & \cdot & \cdot & 0 & 0 & 0\\
1 & 2 & 1 & \cdot & \cdot & \cdot & 0 & 0 & 0\\
0 & 1 & 2 & \cdot & \cdot & \cdot & 0 & 0 & 0\\
\cdot & \cdot & \cdot & \cdot & \cdot & \cdot & \cdot & \cdot & \cdot\\
\cdot & \cdot & \cdot & \cdot & \cdot & \cdot & \cdot & \cdot & \cdot\\
\cdot & \cdot & \cdot & \cdot & \cdot & \cdot & \cdot & \cdot & \cdot\\
0 & 0 & 0 & \cdot & \cdot & \cdot & 2 & 1 & 0\\
0 & 0 & 0 & \cdot & \cdot & \cdot & 1 & 2 & 1\\
0 & 0 & 0 & \cdot & \cdot & \cdot & 0 & 1 & 1\\
\end{matrix}\right].
\label{ecl}
\end{equation}
Therefore, by Lemma \ref{tcb}(b) we see that the eigenvalues of $A_sA_s^T$ are $2+2\cos\frac{2k\pi}{2s+1}$, $k=1,\dotsc,s$. Since these are distinct, $A_sA_s^T$ is diagonalizable. It follows that the eigenvalues of the matrix $(A_sA_s^T)^{m/2}$ are $\left(2+2\cos\frac{2k\pi}{2s+1}\right)^{m/2}$, $k=1,\dotsc,s$.

\parindent12pt
For an $s\times s$ matrix $A$ with characteristic polynomial $p_A$ and eigenvalues $x_1,\dotsc,x_s$, we have
\begin{align}
\det(xI_s+A)&=(-1)^s\det(-xI_s-A)
\nonumber
\\
&=(-1)^sp_A(-x)
\nonumber
\\
&=(-1)^s(-x-x_1)\cdots(-x-x_s)
\nonumber
\\
&=(x+x_1)\cdots(x+x_s).
\label{ecm}
\end{align}  
Then \eqref{eca} follows from \eqref{eck}, \eqref{ecm} and the above identification of the eigenvalues of $A_sA_s^T$.

Consider now the case when $n$ is odd. Then the cylindrical hexagonal graph $H_{m,n}$ looks like in the picture on the right in Figure \ref{fca}. Let $B_n$ be the $n\times (n+1)$ matrix
\begin{equation}
B_n=
\left[\begin{matrix}
1 & 1 & 0 & \cdot & \cdot & \cdot & 0 & 0 & 0\\
0 & 1 & 1 & \cdot & \cdot & \cdot & 0 & 0 & 0\\
0 & 0 & 1 & \cdot & \cdot & \cdot & 0 & 0 & 0\\
\cdot & \cdot & \cdot & \cdot & \cdot & \cdot & \cdot & \cdot & \cdot\\
\cdot & \cdot & \cdot & \cdot & \cdot & \cdot & \cdot & \cdot & \cdot\\
\cdot & \cdot & \cdot & \cdot & \cdot & \cdot & \cdot & \cdot & \cdot\\
0 & 0 & 0 & \cdot & \cdot & \cdot & 1 & 1 & 0\\
0 & 0 & 0 & \cdot & \cdot & \cdot & 0 & 1 & 1\\
\end{matrix}\right].
\label{ecn}
\end{equation}
Then the bi-adjacency matrix of strand $S_i$ is $B_{(n-1)/2}$ for odd $i$, and its transpose $B_{(n-1)/2}^T$ for even $i$. Therefore, since odd-index strands have $(n+1)/2$ black vertices while even-indexed ones have $(n-1)/2$ black vertices, by Theorem \ref{tba}(b) we obtain
\begin{equation}
  \M(H_{m,n};x_1,\dotsc,x_m)=
  x_1x_3\cdots x_{m-1}
  \det\left(x_1\cdots x_m I_{(n+1)/2} + (B_{(n-1)/2}B_{(n-1)/2}^T)^{m/2}\right).
\label{eco}
\end{equation}  
The form of the matrix raised to the power $m/2$ in \eqref{eco} is now
\begin{equation}
B_sB_s^T=
\left[\begin{matrix}
2 & 1 & 0 & \cdot & \cdot & \cdot & 0 & 0 & 0\\
1 & 2 & 1 & \cdot & \cdot & \cdot & 0 & 0 & 0\\
0 & 1 & 2 & \cdot & \cdot & \cdot & 0 & 0 & 0\\
\cdot & \cdot & \cdot & \cdot & \cdot & \cdot & \cdot & \cdot & \cdot\\
\cdot & \cdot & \cdot & \cdot & \cdot & \cdot & \cdot & \cdot & \cdot\\
\cdot & \cdot & \cdot & \cdot & \cdot & \cdot & \cdot & \cdot & \cdot\\
0 & 0 & 0 & \cdot & \cdot & \cdot & 2 & 1 & 0\\
0 & 0 & 0 & \cdot & \cdot & \cdot & 1 & 2 & 1\\
0 & 0 & 0 & \cdot & \cdot & \cdot & 0 & 1 & 2\\
\end{matrix}\right].
\label{ecl}
\end{equation}
Thus, in this case, by Lemma \ref{tcb}(a) we see that the eigenvalues of $B_sB_s^T$ are $2+2\cos\frac{k\pi}{s+1}$, $k=1,\dotsc,s$. These are again distinct, so the eigenvalues of the matrix $(B_sB_s^T)^{m/2}$ are $\left(2+2\cos\frac{k\pi}{s+1}\right)^{m/2}$, $k=1,\dotsc,s$. Proceeding as in part (a) we are again led to formula \eqref{eca}.~\hfill$\square$

\parindent0pt
\medskip
{\it Remark $2$.} The formula provided by Theorem \ref{tca} is a TFK-style formula for a family of honeycomb graphs. This is an unusual situation, as virtually all explicit product formulas in the literature for honeycomb style graphs are ``round formulas,'' in the sense that the size of the factors is linear in the parameters.

\medskip
{\it Remark $3$.} The free energy per site for the family $H_{m,n}$ of honeycomb cylinder graphs turns out to be the same as for toroidal honeycombs, which is known to be maximal. This is in contrast with the family of centrally symmetric honeycombs whose perfect matching enumeration is equivalent to MacMahon's boxed plane partition theorem \cite{MacM}.

\medskip
{\it Remark $4$.} Our graph $H_{m,n}$ has the structure of a nanotube with ``armchair boundary'' (see e.g. \cite{nanotubes}). The other natural type of boundary, zig-zag boundary, is not so interesting from the point of view of perfect matching enumeration, as it simply yields a honeycomb cylinder whose number of perfect matchings is 2 to the number of strands.





\parindent12pt


\section{Cylindric plane partitions and periodic cliffs}

\begin{figure}[t]
\centerline{
\hfill
{\includegraphics[width=0.55\textwidth]{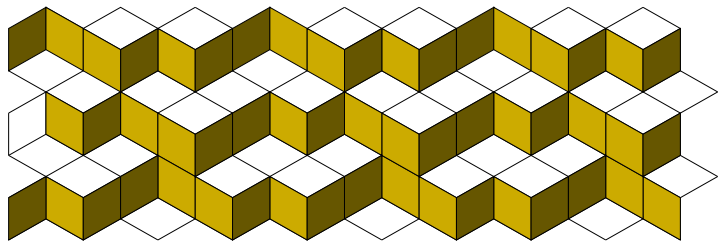}}
\hfill
}
\vskip-0.1in
\caption{\label{fda} An $m$-periodic cliff of height $n$ and horizontal displacement $s$ for $m=3$, $n=3$, $s=4$.}
\end{figure}

\begin{figure}[t]
\vskip0.1in
\centerline{
\hfill
{\includegraphics[width=0.55\textwidth]{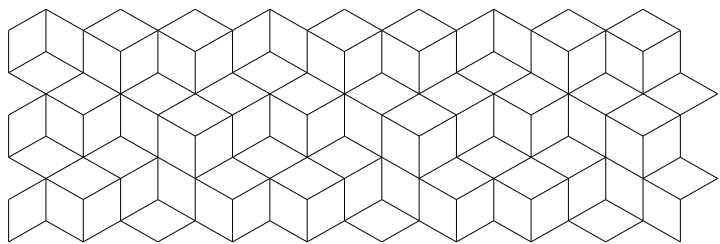}}
\hfill
}
\vskip-0.1in
\caption{\label{fdb}  The corresponding periodic lozenge tiling; southern boundary is translation of northern boundary $2n+s$ units in the polar direction $-\pi/3$.}
\end{figure}

\begin{figure}[t]
\vskip0.1in
\centerline{
\hfill
{\includegraphics[width=0.55\textwidth]{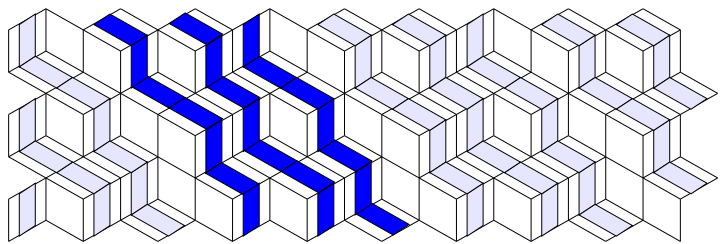}}
\hfill
}
\vskip-0.1in
\caption{\label{fdc} Paths that determine the periodic lozenge tiling; each traverses $s$ horizontal lozenges.}
\end{figure}

\begin{figure}[t]
\vskip0.1in
\centerline{
\hfill
{\includegraphics[width=0.55\textwidth]{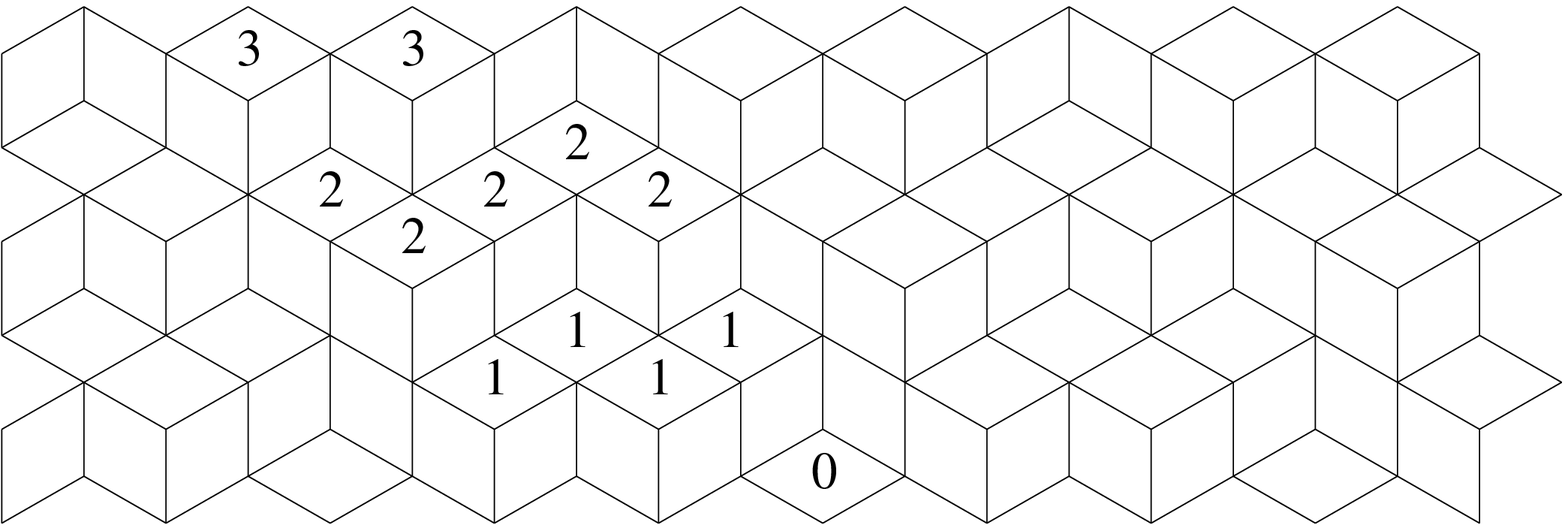}}
\hfill
{\includegraphics[width=0.20\textwidth]{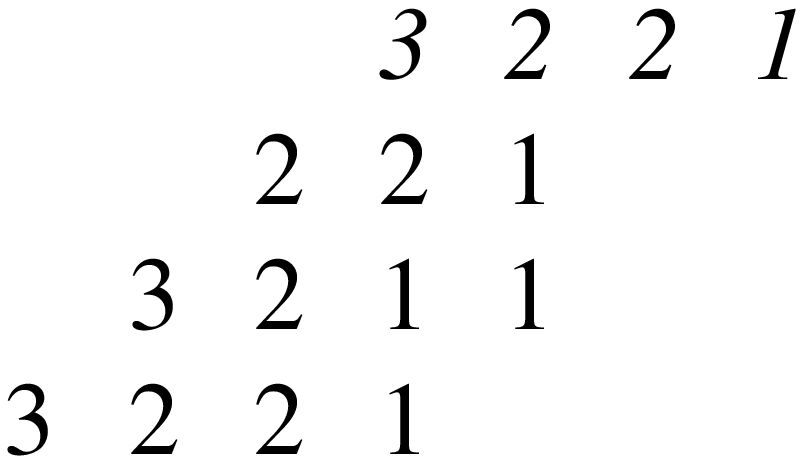}}
\hfill
}
\vskip-0.1in
\caption{\label{fdd} The cliff heights on the horizontal lozenges traversed by the paths determine the paths (left); the corresponding cylindric plane partition (right).}
\end{figure}

\parindent0pt

We define a {\it cliff} to be a stepped surface whose projection on the triangular lattice is bounded by two parallel infinite zigzags (see Figure \ref{fda} for an example). We say that the cliff is {\it $m$-periodic} if it is invariant under translation by $m{\bold v}$, where ${\bold v}$ is the shortest non-zero vector such that translation by it leaves the zigzags invariant. The {\it height} of a cliff is the difference between the heights of the horizontal planes containing its bounding zigzags. A cliff can be viewed as consisting of slices composed of unit cubes cut out by equidistant parallel planes one unit apart; the {\it horizontal displacement} of a cliff is the number of horizontal faces in each such slice. The cliff pictured in Figure \ref{fda} is $3$-periodic, has height 3 and horizontal displacement 4.

\parindent12pt
Cylindric plane partitions were introduced by Gessel and Krattenthaler in \cite{GK}. For Young diagrams $\mu\subset\lambda$ and a positive integer $d$, a cylindric plane partition of shape $\lambda/\mu/d$ is a filling of the skew Young diagram $\lambda/\mu$ with non-negative integers that weakly decrease along rows and columns, with the additional property that when the bottom row is copied above the top row and translated $d$ units to the right, the resulting extended array is still weakly decreasing along columns (see Figure \ref{fdd} for an example).

The main result of this section is the following.

\begin{theo}
\label{tda}
There are as many
$m$-periodic cliffs of height $n$ and horizontal displacement $s$
as 
cylindric partitions of shape $(s+m-1,s+m-2,\dotsc,s)/(m-1,m-2,\dotsc,m)/m$ with entries less or equal than $n$.
The common number is
\begin{equation}
[x^{\lfloor s/2 \rfloor}]\prod_{k=1}^{n+\lfloor s/2 \rfloor}
\left[x+\left(2+2\cos\frac{2k\pi}{2n+s+1}\right)^m\right],
\label{eda}
\end{equation}
where $[x^k]p(x)$ denotes the coefficient of $x^k$ in $p(x)$.

\end{theo}

\begin{figure}[t]
\vskip0.1in
\centerline{
{\includegraphics[width=0.45\textwidth]{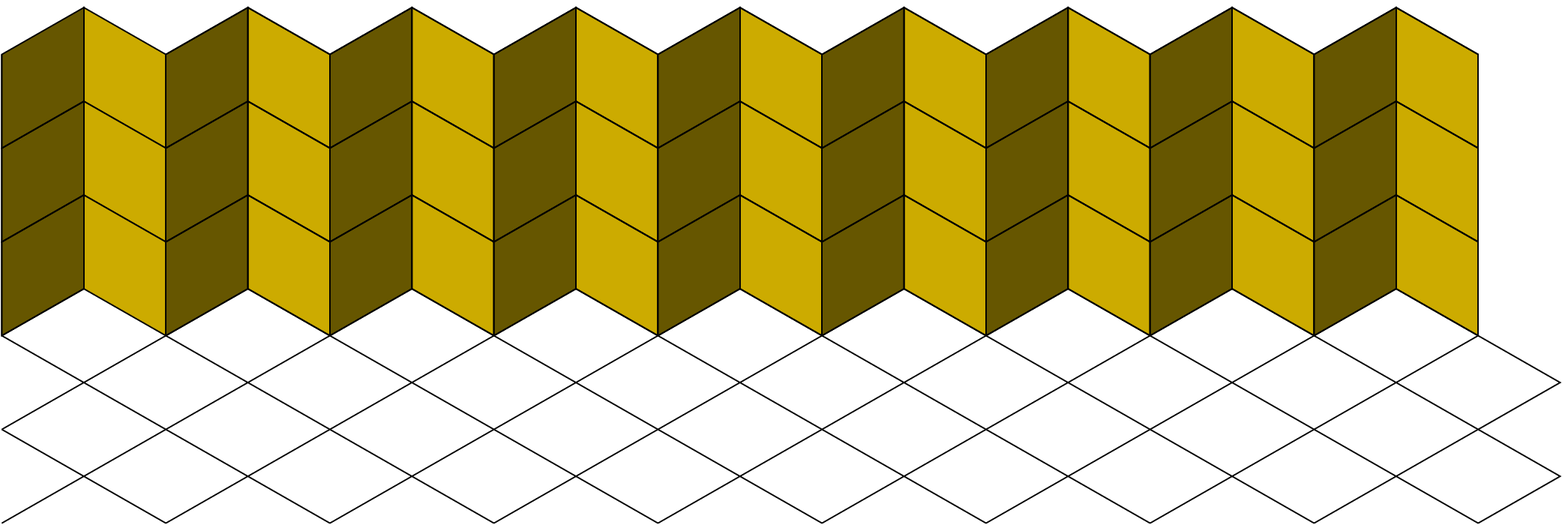}}
\hfill
{\includegraphics[width=0.45\textwidth]{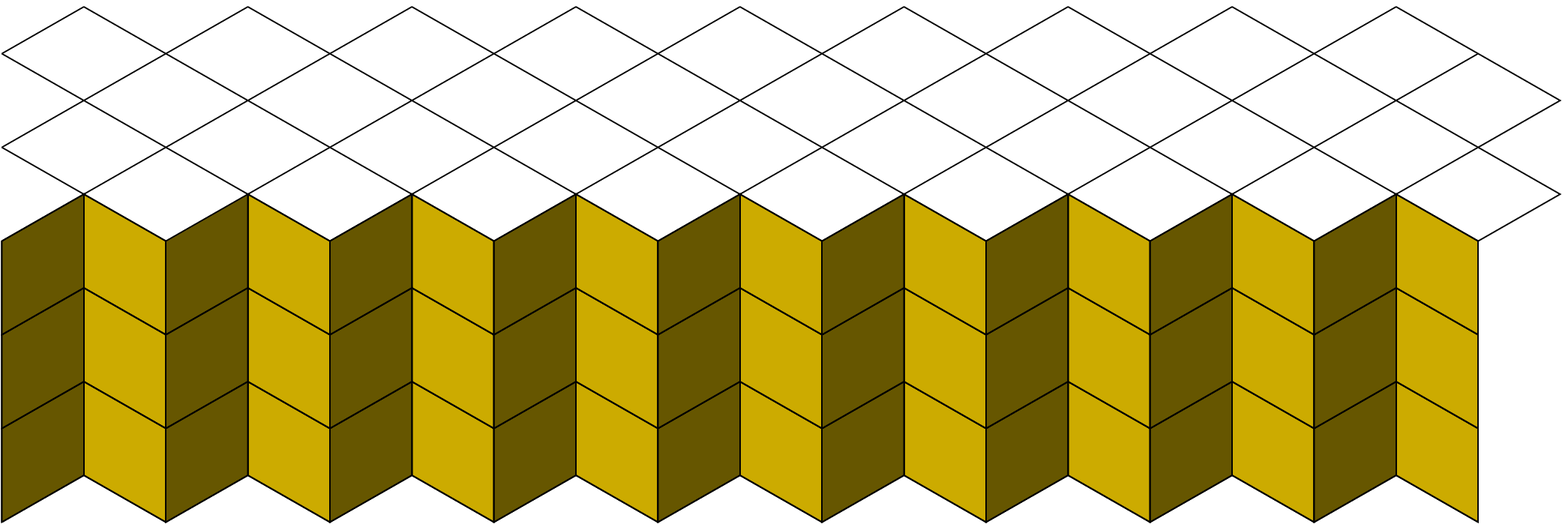}}
}
\vskip-0.1in
\caption{\label{fde} The lowermost (left) and uppermost (right) periodic cliff of height~$n$ and horizontal displacement $s$, for $n=3$, $s=4$.}
\end{figure}

\parindent0pt
\medskip
{\it Remark $5$.} The result in Theorem \ref{tda} seems to be the first explicit product formula in the literature for the number of cylindric partitions of a given shape and a given bound on the size of its entries.

\parindent12pt
It can be interpreted as giving the number of $m$-periodic stepped surfaces that fit in the infinite tube enclosed by the lowermost and uppermost surfaces shown in Figure \ref{fde}.
From this point of view, it is a counterpart of MacMahon's boxed plane partition theorem \cite{MacM} (and as far as the container shape is concerned, it resembles even more Proctor's theorem on the number of plane partitions that fit in a right prism with staircase shape base \cite{Proc}). 

\medskip

The proof will follow from Theorem \ref{tca} and Proposition \ref{tdb} below. Define the zigzag strip~$Z_k$ to be the region on the triangular lattice\footnote{ Drawn so that one family of lattice lines is vertical.} between two infinite horizontal zigzags, the lower being the translation of the upper $k$ units in the polar direction $-\pi/3$.

An edge of the triangular lattice parallel to the polar direction $-\pi/3$ is called a ``/''-edge. Given a lozenge tiling $T$ of $Z_k$ and a ``/''-edge $e$ on its top boundary, start at $e$ and follow lozenges containing a ``/''-edge in their boundary until a ``/''-edge is reached on the bottom boundary of $Z_k$. The resulting sequence of lozenges is called a {\it path of lozenges}.

\parindent12pt

\begin{prop}
\label{tdb}
The following are equinumerous:

$(1)$ $m$-periodic cliffs of height $n$ and horizontal displacement $s$

$(2)$ $m$-periodic lozenge tilings of the zigzag strip $Z_{2n+s}$ with $s$ horizontal lozenges along each path of lozenges 

$(3)$ perfect matchings of the honeycomb cylinder $H_{2m,2n+s}$ with $ms$ vertical edges 

$(4)$ $m$-periodic families of non-intersecting paths of lozenges in $Z_{2n+s}$  with $s$ horizontal lozenges on each path

$(5)$ cylindric partitions of shape $(s+m-1,s+m-2,\dotsc,s)/(m-1,m-2,\dotsc,m)/m$ with entries less or equal than $n$

\end{prop}

\parindent12pt

\begin{proof} Consider a rectangular system of coordinates in which the faces of the stepped surface are parallel to the coordinate planes, and so that when viewed along the line $x=y=z$, the faces are seen as congruent rhombi with angles of $60^\circ$ and $120^\circ$. Projecting the stepped surface on a plane perpendicular to the line $x=y=z$ shows that the sets (1) and (2) are in one-to-one correspondence (see also \cite{DT}).

 Any lozenge tiling of the zigzag strip $Z_{2n+s}$ naturally defines an infinite family of paths of lozenges, obtained by starting at the ``/''-segments of the upper boundary, following along lozenges in the tiling, and ending at the ``/''-segments of the lower boundary. Since these paths come from a tiling, they are non-intersecting. This implies that paths starting at consecutive segments of the upper boundary end at consecutive segments of the lower boundary. This in turn implies that all these paths contain the same number of horizontal lozenges.
Given an $m$-periodic lozenge tiling of $Z_{2n+s}$ in which this common number is $s$, associate to it the previously described family of paths of lozenges. This is a bijection between sets (2) and (4). 

For a bijection between sets (2) and (3), note that $m$-periodic tilings of $Z_{2n+s}$ can be identified with tilings of the quotient of $Z_{2n+s}$ under the action of the horizontal translation that leaves them invariant. The dual of the quotient region is precisely the honeycomb cylinder $H_{2m,2n+s}$. Furthermore, if in the tiling each of the described paths of lozenges has $s$ horizontal lozenges, the total number of horizontal lozenges in the quotient region (which correspond to vertical edges in $H_{2m,2n+s}$) is $ms$. 

The bijection between (4) and (5) is indicated in Figure \ref{fdd}. Simply note that the whole $m$-periodic family of paths of lozenges is determined by $m$ consecutive paths, which in turn are determined by the sequences of heights of the horizontal lozenges in them. These form $m$ weakly decreasing sequences of non-negative integers. Arranging them in an array from bottom to top, with each successive row one unit further to the right and omitting the 0's, one obtains a cyclic partition of shape $(s+m-1,s+m-2,\dotsc,s)/(m-1,m-2,\dotsc,m)/m$. \end{proof}  

\parindent0pt

{\it Proof of Theorem $\ref{tda}$.} By Proposition \ref{tdb}, both the number of $m$-periodic cliffs of height $n$ and horizontal displacement $s$, and the number of cylindric partitions of shape $(s+m-1,s+m-2,\dotsc,s)/(m-1,m-2,\dotsc,m)/m$ with entries less or equal than $n$, is equal to the number of perfect matchings of the cylindrical honeycomb $H_{2m,2n+s}$ with $ms$ vertical edges.

\parindent12pt
Denote by $M(H_{m,n};x)$ the sum of the weights of the perfect matchings of $H_{m,n}$ when all vertical edges are weighted by $x$, and all other edges by 1. If $s$ is even, Theorem \ref{tca} gives 
\begin{equation}
M(H_{2m,2n+s};x)=
\prod_{k=1}^{n+s/2}\left[x^{2m}+\left(2+2\cos{\frac{2k\pi}{2n+s+1}}\right)^{m}\right].
\label{edb}
\end{equation}
The number of perfect matchings of $H_{2m,2n+s}$ with $ms$ vertical edges is then the coefficient of $x^{ms}=(x^{2m})^{s/2}$ in the product \eqref{edb}, which agrees with the right hand side of \eqref{eda}.

If on the other hand $s$ is odd, Theorem \ref{tca} gives 
\begin{equation}
M(H_{2m,2n+s};x)=x^m
\prod_{k=1}^{n+(s-1)/2}\left[x^{2m}+\left(2+2\cos{\frac{2k\pi}{2n+s+1}}\right)^{m}\right].
\label{edc}
\end{equation}
The number of perfect matchings $H_{2m,2n+s}$ with $ms$ vertical edges is now the coefficient of $x^{ms-m}=(x^{2m})^{(s-1)/2}$ in the product on the right hand side of \eqref{edc}, which again agrees with the right hand side of \eqref{eda}.
\epf





\section{Square cylinder graphs of even girth}

\begin{figure}[t]
\centerline{
\hfill
{\includegraphics[width=0.35\textwidth]{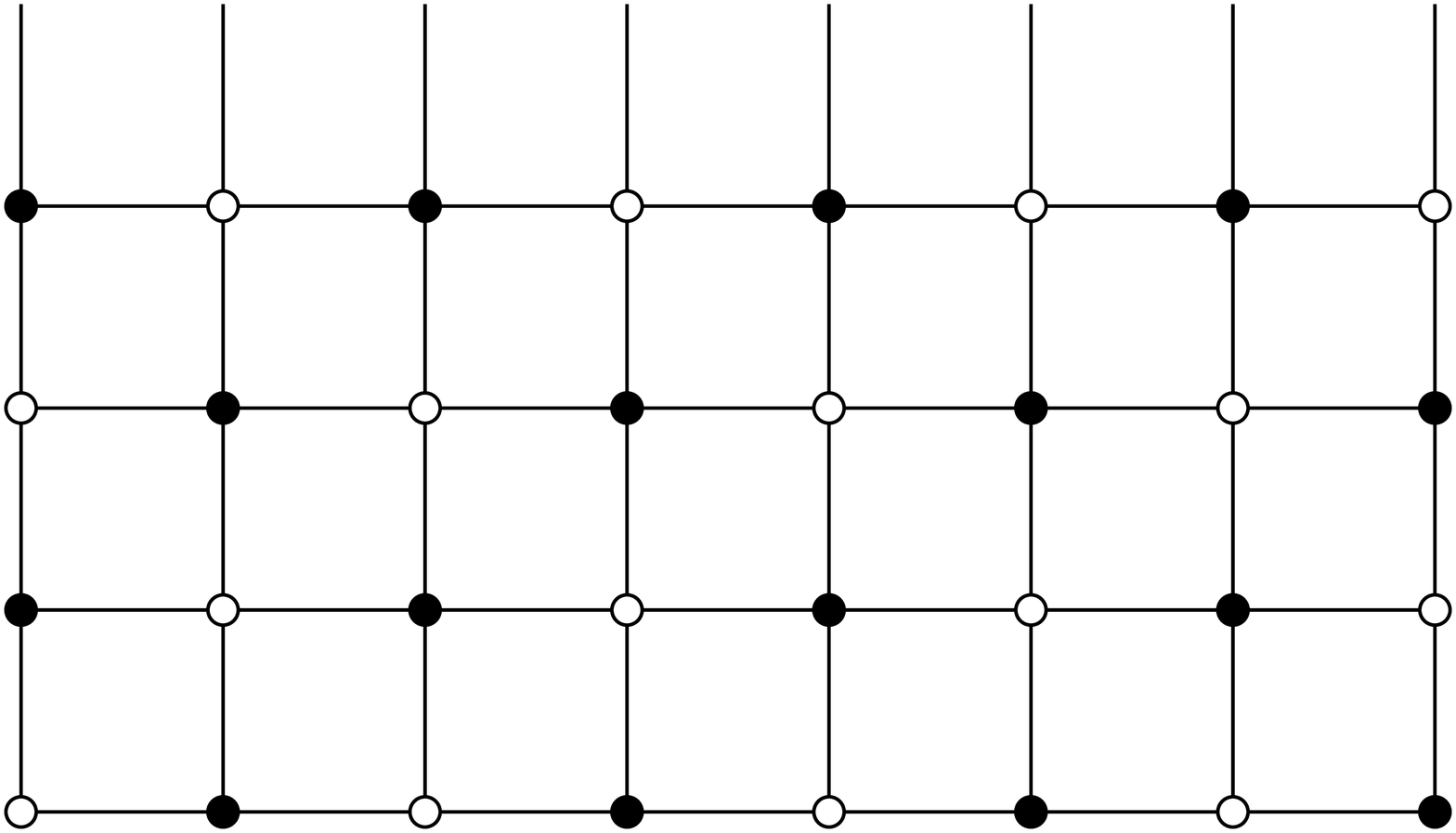}}
\hfill
{\includegraphics[width=0.35\textwidth]{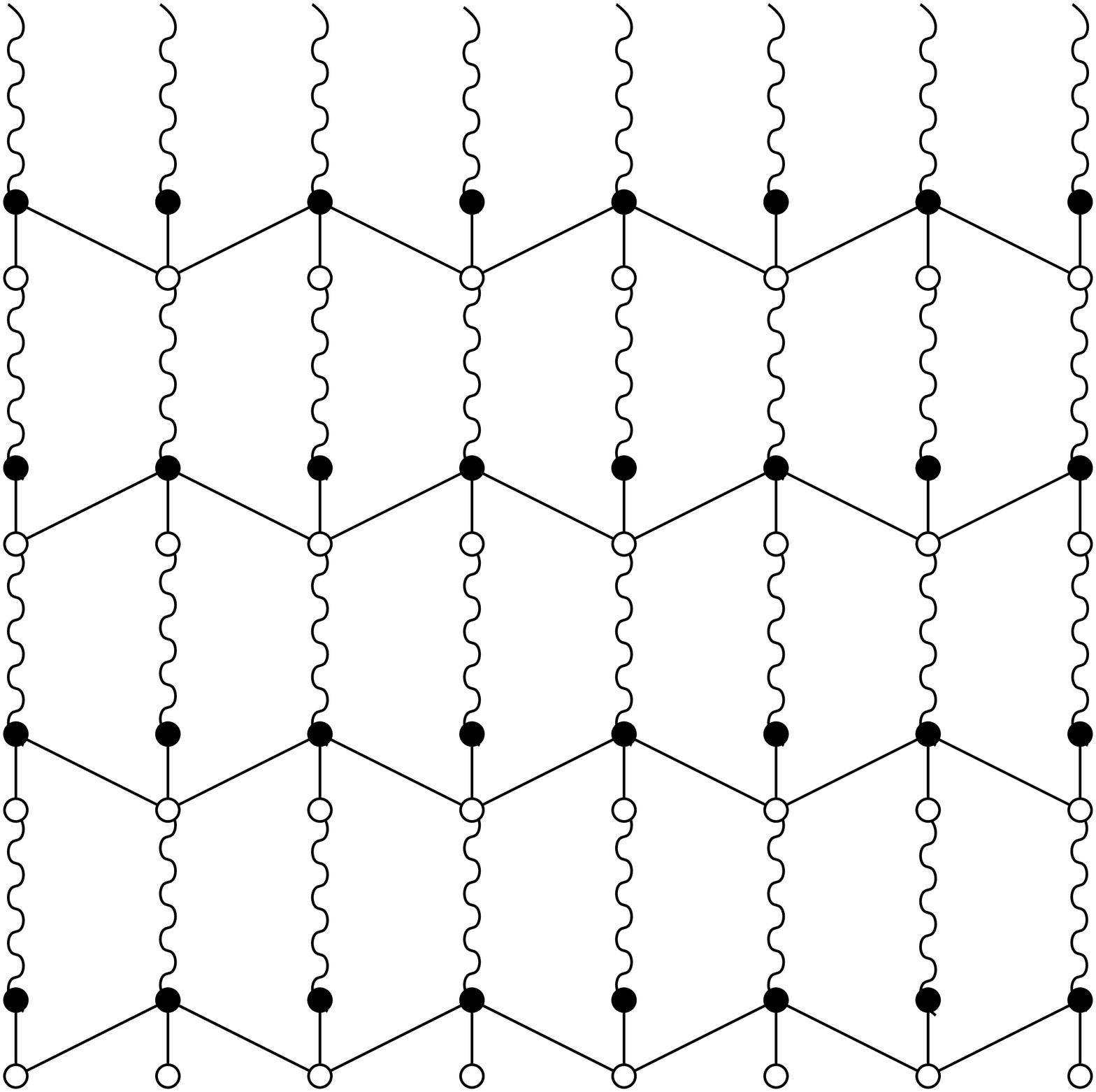}}
\hfill
}
\caption{\label{fea} The square cylinder graph $C_{m,n}$ for $m=4$, $n=8$
(left). A fabric graph with the same number of perfect matchings
(right). In both pictures, edges extending up from the top vertices connect down to the corresponding bottom vertices.}
\end{figure}

Define the {\it square cylinder graph} $C_{m,n}$ to be the graph obtained from the $m\times n$ rectangular grid graph with vertices $\{(i,j): i,j\in\Z, 0\leq i\leq m-1,  0\leq j\leq n-1\}$ by adding an edge connecting vertex $(i,0)$ to vertex $(i,n-1)$ for each $i=1,\dotsc,m$ (see Figure \ref{fea} for an example). The {\it girth} of the square cylinder graph $C_{m,n}$ is defined to be equal to $m$.

The main result of this section is the following.

\begin{theo}
\label{tea}
If $m$ is even, we have
\begin{align}
\M(C_{m,n})&=
\nonumber
\\[10pt]
&\!\!\!\!\!\!\!\!\!\!\!\!\!\!\!\!\!\!\!\!\!\!\!\!
\sqrt{
\prod_{k=1}^n
\left[1+\left(\!\cos\frac{k\pi}{n+1}+\sqrt{1+\cos^2\frac{k\pi}{n+1}}\,\right)^{\!\!m}\,\right]
\!\left[1+\left(\!\cos\frac{k\pi}{n+1}-\sqrt{1+\cos^2\frac{k\pi}{n+1}}\,\right)^{\!\!m}\,
\right]
}
\nonumber
\\[10pt]
&\!\!\!\!\!\!\!\!\!\!\!\!\!\!\!\!\!\!\!\!\!\!\!\!
=2^{n-2\lfloor n/2\rfloor}
\prod_{k=1}^{\lfloor n/2\rfloor}
\left[1+\left(\!\cos\frac{k\pi}{n+1}+\sqrt{1+\cos^2\frac{k\pi}{n+1}}\,\right)^{\!\!m}\,\right]
\!\left[1+\left(\!\cos\frac{k\pi}{n+1}-\sqrt{1+\cos^2\frac{k\pi}{n+1}}\,\right)^{\!\!m}\
\right].
\label{eea}
\end{align}

\end{theo}

\parindent0pt
Our proof of Theorem \ref{tea} employs the following result.

\parindent12pt

\begin{lem}
\label{teb}  
Let $A=(a_{ij})_{i,j\geq1}$ be the infinite matrix having support
\begin{equation}
A=\left[\begin{matrix}
1 & 1 &  &  &  &  & & &  & \\
 & 1 &  &  &  &  & & &  & \\
 & 1 & 1 & 1 &  &  & & &  & \\
 &  &  & 1 &  &  & & &  & \\
 &  &  & 1 & 1 & 1 & & &  & \\
 &  &  &  &  & 1 & & &  & \\
 &  &  &  &  & 1 &1 & &  & \\
 &  &  &  & &  &  &\cdot &  & \\
 &  &  &  & &  &  & &\cdot  & \\
 &  &  &  & &  &  & &  & \cdot\\
\end{matrix}\right],
\label{eeb}
\end{equation}
and define $A_n$ to be its restriction to the first $n$ rows and first $n$ columns. Then the eigenvalues of $A_nA_n^T$ are
\begin{align}
&\left(\cos\frac{k\pi}{n+1}+\sqrt{1+\cos^2\frac{k\pi}{n+1}}\right)^2,\ \ \ k=1,2,\dotsc,\lceil n/2\rceil
\nonumber
\\[10pt]
&\left(\cos\frac{k\pi}{n+1}-\sqrt{1+\cos^2\frac{k\pi}{n+1}}\right)^2,\ \ \ k=1,2,\dotsc,\lceil n/2\rceil,
\label{eec}
\end{align}
with the convention that for $n$ odd and $k=\lceil n/2\rceil$, the two expressions above --- both of which equal $1$ --- supply the eigenvalue $1$ with a total multiplicity of one.

\end{lem}

\begin{proof} The matrix $A_nA_n^T$ has order $n$, and the pattern of its non-zero entries looks slightly different depending on whether the index $n$ is even or odd. For even indices we have
\setcounter{MaxMatrixCols}{20}
\begin{equation}
A_{2n}A_{2n}^T=\begin{bmatrix}
2&1&1& & & & & & & & & \\
1&1&1& & & & & & & & & \\
1&1&3&1&1& & & & & & & \\
 & &1&1&1& & & & & & & \\
 & &1&1&3& & & & & & & \\
 & & & & &\cdot& & & & & & \\
 & & & & & &\cdot& & & & & \\
 & & & & & & &\cdot& & & & \\
 & & & & & & & &3&1&1& \\
 & & & & & & & &1&1&1& \\
 & & & & & & & &1&1&3&1\\
 & & & & & & & & & &1&1\\
\end{bmatrix},
\label{eed}
\end{equation}
(where only the pattern followed by the non-zero entries is indicated). For odd indices the truncation at the bottom right corner is like in the matrix obtained from the one above by deleting the last row and column.

However, the proof we present works for both even and odd indices.
The reason is that all the recurrences we work with are obtained by expanding the resulting determinants along their first rows or columns, and the pattern of the non-zero entries of $A_nA_n^T$ is the same around the top left corner for both $n$ even and $n$ odd. We give here the details for the even case.

Denote by $q_n(\lambda)$ the characteristic polynomial of $A_{2n}A_{2n}^T$:
\begin{equation}
q_n(\lambda)=\det\begin{bmatrix}
\lambda-2&-1&-1& & & & & & & & & \\
-1&\lambda-1&-1& & & & & & & & & \\
-1&-1&\lambda-3&-1&-1& & & & & & & \\
 & &-1&\lambda-1&-1& & & & & & & \\
 & &-1&-1&\lambda-3& & & & & & & \\
 & & & & &\cdot& & & & & & \\
 & & & & & &\cdot& & & & & \\
 & & & & & & &\cdot& & & & \\
 & & & & & & & &\lambda-3&-1&-1& \\
 & & & & & & & &-1&\lambda-1&-1& \\
 & & & & & & & &-1&-1&\lambda-3&-1\\
 & & & & & & & & & &-1&\lambda-1\\
\end{bmatrix}.
\label{eee}
\end{equation}
Regarding the first column as $(\lambda-3),-1,-1,0,\dotsc)^T+(1,0,0,\dotsc,0)^T$ and using the linearity of the determinant in columns, we obtain
\begin{equation}
q_n(\lambda)=p_n(\lambda)+r_n(\lambda),
\label{eef}
\end{equation}
with
\begin{equation}
p_n(\lambda)=\det\begin{bmatrix}
\lambda-3&-1&-1& & & & & & & & & \\
-1&\lambda-1&-1& & & & & & & & & \\
-1&-1&\lambda-3&-1&-1& & & & & & & \\
 & &-1&\lambda-1&-1& & & & & & & \\
 & &-1&-1&\lambda-3& & & & & & & \\
 & & & & &\cdot& & & & & & \\
 & & & & & &\cdot& & & & & \\
 & & & & & & &\cdot& & & & \\
 & & & & & & & &\lambda-3&-1&-1& \\
 & & & & & & & &-1&\lambda-1&-1& \\
 & & & & & & & &-1&-1&\lambda-3&-1\\
 & & & & & & & & & &-1&\lambda-1\\
\end{bmatrix}
\label{eeg}
\end{equation}
(where the matrix has $2n$ rows) and
\begin{equation}
r_n(\lambda)=\det\begin{bmatrix}
\lambda-1&-1& & & & & & & & & \\
-1&\lambda-3&-1&-1& & & & & & & \\
 &-1&\lambda-1&-1& & & & & & & \\
 &-1&-1&\lambda-3& & & & & & & \\
 & & & &\cdot& & & & & & \\
 & & & & &\cdot& & & & & \\
 & & & & & &\cdot& & & & \\
 & & & & & & &\lambda-3&-1&-1& \\
 & & & & & & &-1&\lambda-1&-1& \\
 & & & & & & &-1&-1&\lambda-3&-1\\
 & & & & & & & & &-1&\lambda-1\\
\end{bmatrix}
\label{eeh}
\end{equation}
(where the matrix has $2n-1$ rows). Expanding the determinant on the right hand side of \eqref{eeg} along the first column yields three non-zero terms. The first is clearly $(\lambda-3)r_n(\lambda)$. The second is readily seen (by expanding along its first column) to equal $-p_{n-1}(\lambda)-r_{n-1}(\lambda)$. Similarly, expanding along the first column in the third term, one sees that it is equal to $-\lambda r_{n-1}(\lambda)$. This gives
\begin{equation}
p_n(\lambda)=(\lambda-3)r_{n}(\lambda)-p_{n-1}(\lambda)-(\lambda+1)r_{n-1}(\lambda).
\label{eei}
\end{equation}
On the other hand, expanding along the first column in the determinant for $r_n(\la)$, we get
\begin{equation}
r_n(\la)=(\la-1)p_{n-1}(\la)-r_{n-1}(\la).
\label{eej}
\end{equation}
The latter gives
\begin{equation}
p_{n-1}=\frac{1}{\la-1}(r_n+r_{n-1}),
\label{eejj}
\end{equation}
which when substituted into \eqref{eei} yields
\begin{equation}
r_{n+1}-(\la^2-4\la+1)r_{n}+\la^2r_{n-1}=0.
\label{eek}
\end{equation}
This recurrence holds for $n\geq1$ if we define $r_0=0$. The solutions of the characteristic equation of recurrence \eqref{eek} are
\begin{equation}
\theta_{1,2}=\frac{\la^2-4\la+1\pm(\la-1)\sqrt{\la^2-6\la+1}}{2}.
\label{eel}
\end{equation}
Therefore, $r_n$ can be expressed as
\begin{equation}
r_{n}=c_1\theta_1^n+c_2\theta_2^n,\ \ \ n\geq0,
\label{eem}
\end{equation}
where the coefficients $c_1$ and $c_2$ can be determined from the initial conditions $r_0=0$, $r_1=\la-1$.

Substituting the expression \eqref{eem} for $r_n$ into \eqref{eejj}, and then using the resulting expression for $p_n$ and the expression \eqref{eem} for $r_n$ in equation \eqref{eef}, we obtain
\begin{equation}
(\la-1)q_{n}(\la)=c_1(\la+\theta_1)\theta_1^n+c_2(\la+\theta_2)\theta_2^n.
\label{een}
\end{equation}
Using the initial conditions $r_0=0$, $r_1=\la-1$, one readily gets from \eqref{eem} that
\begin{equation}
c_1=\frac{1}{\sqrt{\la^2-6\la+1}}, \ \ \ c_2=-\frac{1}{\sqrt{\la^2-6\la+1}}.
\label{eeo}
\end{equation}
Our goal is to find the eigenvalues of $A_{2n}A_{2n}^T$, which are the zeros of $q_{n}(\la)$. By \eqref{een} and \eqref{eeo}, the zeros of $(\la-1)q_{n}(\la)$ are those values of $\la$ for which
\begin{equation}
\frac{1}{\sqrt{\la^2-6\la+1}}\left[(\la+\theta_1)\theta_1^n-(\la+\theta_2)\theta_2^n\right]=0.
\label{eep}
\end{equation}
Using \eqref{eel}, one sees after some manipulation that $(\la+\theta_1)\theta_1^n=(\la+\theta_2)\theta_2^n$ if and only if
\begin{equation}
(\la-1)(\theta_1^{n+1}-\la\theta_1^{n})=0.
\label{eepp}
\end{equation}
Since $\theta_1\theta_2=\la^2$, if $\la\neq1$ this amounts to $\theta_1^{2n+1}=\la^{2n+1}$, which in turn means that
\begin{equation}
\theta_1=\varepsilon\la, \ \ \ \varepsilon=e^{\frac{2k\pi i}{2n+1}},\ \ \ k=0,1,\dotsc,2n.
\label{eeq}
\end{equation}
Expressing $\theta_1$ by formula \eqref{eel}, \eqref{eeq} reduces to a quadratic equation in lambda, with solutions
\begin{equation}
\frac{1}{2\varepsilon}\left(1+4\varepsilon+\varepsilon^2\pm\sqrt{1+8\ve+14\ve^2+8\ve^3+\ve^4}\right).
\label{eer}
\end{equation}
Using $\ve+\ve^{-1}=2\cos\frac{2k\pi}{2n+1}$ and then the formula $\cos 2\alpha=2\cos^2\alpha-1$, the expression \eqref{eer} can be transformed into
\begin{equation}
\left(\cos\frac{k\pi}{2n+1}\pm\sqrt{1+\cos^2\frac{k\pi}{2n+1}}\right)^2.
\label{ees}
\end{equation}
By the paragraph before \eqref{eep}, we obtain that the zeros of $q_n(\la)$ are among
\begin{align}
\left(\cos\frac{k\pi}{2n+1}\pm\sqrt{1+\cos^2\frac{k\pi}{2n+1}}\right)^2,\ \ \ k=0,1,\dotsc, 2n.
\label{eet}
\end{align}
Since $\cos(\pi-x)=-\cos x$, all the distinct values provided by \eqref{eet} are obtained if $k$ runs over $0,1,\dotsc,n$. Furthermore, for $k=0$, these become $3\pm2\sqrt{2}$, which are the values that make the denominator in \eqref{eep} equal to zero. It is not hard to see that the remaining $2n$ values (obtained by taking $k=1,2,\dotsc,n$ in \eqref{eet}) are all distinct. Therefore they are the eigenvalues of $q_n(\lambda)$, as claimed by the statement of the lemma.

For the case of odd index, the above arguments applied to the matrix $A_{2n-1}A_{2n-1}^T$ lead to the conclusion that the zeros of its characteristic polynomial multiplied by $\la-1$ form the set
\begin{align}
\left(\cos\frac{k\pi}{2n}\pm\sqrt{1+\cos^2\frac{k\pi}{2n}}\right)^2,\ \ \ k=1,\dotsc,n.
\label{eeu}
\end{align}
These are $2n-1$ distinct values, the only repetition being for $k=n$, when both choices of the sign lead to the value 1. Thus the eigenvalues of $A_{2n-1}A_{2n-1}^T$ are indeed the ones described in the statement of the lemma. \end{proof}

{\it Proof of Theorem $\ref{tea}$.} Since $m$ is even, using the vertex splitting lemma of \cite{FT} (see Lemma~1.3 there) one can readily construct a fabric graph $G_{m,n}$ with the same number of perfect matchings as the square cylinder graph $C_{m,n}$ --- the fabric graph corresponding to the square cylinder on the left in Figure \ref{fea} is shown on the right in the same figure\footnote{ The indicated construction does not work if $m$ is odd, as there is a periodicity involving every two consecutive strands in the picture on the right in Figure \ref{fea}.}. The bi-adjacency matrix of odd-indexed strands in $G_{m,n}$ is the matrix $A_n$ from Lemma \ref{teb}, and the bi-adjacency matrix of even-indexed strands in $G_{m,n}$ is its transpose $A_n^T$. Therefore, by Theorem \ref{tba}(b) we have
\begin{equation}
\M(C_{m,n})=\M(G_{m,n})=\det\left(I_n+(A_nA_n^T)^{m/2}\right).
\label{eev}
\end{equation}
Then formula \eqref{eea} follows from \eqref{ecm} and Lemma \ref{teb}. \hfill $\square$


\parindent0pt
\medskip
{\it Remark $6$.} Note that despite the great similarity between our square cylinder graphs and the rectangular grid graphs, our formula from Theorem \ref{tea} has a quite different form compared to the TFK formula (see \eqref{eff}). A partial connection between the two is established in Section 7.

\parindent12pt

%
%
%


\section{A factorization theorem for perfect matchings of symmetric turn-bipartite graphs}

In this section we give an extension of our factorization theorem of \cite{FT} that applies to certain planar non-bipartite graphs, which we call turn-bipartite.

For a planar graph $G$ embedded in a region $R$ with holes, we say that $G$ is {\it turn-bipartite} if
the length of any cycle $C$ of $G$ has the same parity as the number of holes of $R$ that are in the interior of $C$. The picture on the left in Figure \ref{ffa} shows a turn-bipartite graph embedded in an annulus.

Following the terminology introduced in \cite{FT}, we say that a plane graph $G$ is {\it symmetric} if it is invariant under the reflection across some straight line.
Clearly, a symmetric graph has no perfect matching unless the axis of symmetry contains an even number of vertices (otherwise, the total number of vertices is odd); we will assume this throughout this section.

A {\it weighted} symmetric graph is a symmetric graph with a weight function on the edges that is constant on the orbits of the reflection. The {\it width} of a symmetric graph $G$, denoted $\w(G)$, is defined to be half the number  of vertices of $G$ lying on the symmetry axis.

Let $G$ be a weighted symmetric graph with symmetry axis $\ell$, which we consider to be horizontal. Let $a_1, b_1, a_2,b_2,\dotsc,a_{\w(G)},b_{\w(G)}$ be the vertices lying on $\ell$, as they occur from left to right.
A {\it reduced} subgraph of $G$ is a graph obtained from $G$ by deleting at each vertex $a_i$ either all  incident edges above $\ell$ (we refer to this operation for short as ``cutting above $a_i$") or all incident edges below $\ell$ (``cutting below $a_i$," for short). We recall the following result proved in~\cite{FT} (see Lemma 1.1 there).



\begin{lem}[\cite{TF}]
\label{tga}  
All $2^{\w(G)}$ reduced subgraphs of a weighted symmetric graph $G$ have the same weighted count of perfect matchings.
\end{lem}

Let $G$ be a turn-bipartite weighted symmetric graph with symmetry axis $\ell$. Assume that all the holes in the region $R$ in which $G$ is embedded are along $\ell$. Then the subgraph $G_{\geq}$ induced by the vertices of $G$ on or above $\ell$ is bipartite. 

Let us color the vertices in the two bipartition classes of $G_{\geq}$ black and white.
For definiteness, choose $a_1$ to be white.
We define a subgraph of $G$ as follows. Perform cutting operations above all white $a_i$'s and black $b_i$'s, and below all black $a_i$'s and white $b_i$'s. Note that this procedure yields cuts of the same kind at the endpoints of each edge lying on $\ell$.  Reduce the weight of each such edge by half; leave all other weights unchanged. Denote by $G'$ the resulting graph.

The main result of this section is the following.

\begin{theo}
\label{tgb}
Let $G$ be a weighted, turn-bipartite symmetric graph  embedded in a region $R$ so that all the holes of $R$ are along the symmetry axis. Then 
\begin{equation}
\M(G)=2^{\w(G)}\M(G').
\label{ega}
\end{equation}

\end{theo}

Note that in the special case when the region $R$ is simply connected, $G$ is bipartite, and the above result becomes the factorization theorem \cite[Theorem 1.2]{FT}.


Our proof will follow from Lemma \ref{tga} and the following two additional lemmas.  A {\it doubly reduced} subgraph of $G$ is a graph obtained from $G$ by cutting either above or below each $a_i$ and $b_i$, $i=1,\dotsc,\w(G)$.

\begin{lem}
\label{tgc}
Let $G$ be a symmetric turn-bipartite graph. 

$(${\rm a}$)$. If $G$ is not bipartite, then $G$ must have a symmetric odd cycle.

$(${\rm b}$)$. All doubly reduced subgraphs of $G$ are bipartite.


\end{lem}

\begin{proof}
$(${\rm a}$)$. 
As $G$ is not bipartite, it has an odd cycle $C$.
Since $C$ is odd and $G$ is turn-bipartite, $C$ goes around an odd number of holes of~$R$, so in particular $C$ must cross the symmetry axis~$\ell$.
Let $C'$ be the mirror image of $C$ across $\ell$. Let $\bar{C}$ be the cycle consisting of the edges of the unbounded face of the graph $C\cup C'$. Then $\bar{C}$ is symmetric, and goes around precisely those holes of the region $R$ that $C$ goes around. Since $G$ is turn-bipartite and $C$ is odd, $\bar{C}$ is also odd.

$(${\rm b}$)$. Suppose $C$ is an odd cycle of $G$. We claim that $C$ must have at least one vertex on $\ell$. Indeed, consider the symmetric cycle $\bar{C}$ from the proof of part (a). By construction, $\bar{C}$ and $C$ have the same set of vertices on $\ell$. But if $\bar{C}$ does not have any vertex on $\ell$, then all vertices on it come in symmetric pairs, which implies that $\bar{C}$ has an even number of vertices, a contradiction.

Since the cutting operations involved in the definition of each doubly reduced subgraph $H$ guarantee that all such cycles $C$ of $G$ are interrupted, it follows that there is no odd cycle in~$H$.
\end{proof}

%

\begin{lem}
\label{tgd}
Let $G$ be a connected, symmetric turn-bipartite graph which is not bipartite. Set~$k=\w(G)$. Assume that $\{a_1,\dotsc,b_k\}$ is an independent set, and that each vertex in it has precisely one incident edge from above and one from below.

$(${\rm a}$)$. Let $u$ and $u'$ be vertices of $G$ that are mirror images across $\ell$. Then for any doubly reduced subgraph $H$ of $G$, $u$ and $u'$ have opposite colors in $H$.

$(${\rm b}$)$. Let $x$ be a fixed reference vertex in $G\setminus\{a_1,\dotsc,b_k\}$, and let $H$ be a doubly reduced subgraph of $G$. Color the vertices in the bipartition classes of $H$ black and white so that $x$ is white. Then:

\ \ $(i)$ for any vertex $v\in G\setminus\{a_1,\dotsc,b_k\}$, the color of $v$ is uniquely determined, and is\linebreak\phantom{\ \ $(i)$} independent of $H$.

\ \ $(ii)$ the color of each $a_i$ and $b_i$ is determined by whether the cut above or the cut below\linebreak\phantom{\ \ $(i)$} them was made to obtain $H$.


\end{lem}

\begin{proof} Note that $G$ connected implies that $G_{\geq}$ is connected. Indeed, let $u$ and $v$ be two vertices of $G$ on or above $\ell$. Since $G$ is connected, there is a path in $G$ that connects $u$ to $v$. Reflecting any portion of $P$ that is below $\ell$ across the symmetry axis yields a path in $G_{\geq}$ connecting $u$ to~$v$.

Note also that this in turn implies that the subgraph graph $G_{>0}$ induced by the vertices of $G$ above $\ell$ is connected. Indeed, the path $P$ in $G_{\geq}$ between two vertices above $\ell$ cannot visit a vertex on $\ell$, otherwise the unique neighbor in $G_{\geq}$ of that vertex would be visited twice by $P$.

(a). By Lemma \ref{tgc}(a), $G$ has a symmetric odd cycle $C$. Let $c$ and $c'$ be mirror image vertices on~$C$, with $c$ above $\ell$. Since $G_{>0}$ is connected, there is a  path $P$ in it from $u$ to $c$. Follow $P$ from $u$ to $c$, then follow $C$ to $c'$ choosing the odd arc (the sum of the lengths of the two arcs of $C$ between $c$ and $c'$ is odd, so one of them is odd). Travel from $c'$  to $u'$ along the mirror image $P'$ of $P$. This is a walk from $u$ to $u'$ with an odd number of edges, and since it does not contain any vertex on $\ell$ ($P$ and $P'$ clearly don't, and the odd arc of $C$ doesn't either, otherwise by symmetry it would have even length), this walk is contained in all doubly reduced subgraphs of $G$.

(b). $(i)$. Suppose first that $v$ is above $\ell$.  As $G_{>0}$ is connected, we can choose a path $P$ in it connecting $x$ to $v$. Clearly, $P$ is included in $H$, and is the same for all doubly reduced subgraphs $H$. Since the parity of its length determines whether $v$ and $x$ have the same or opposite color in $H$, this proves the statement in this case.

If $v$ is below $\ell$, use the same argument, but with $x$ replaced by its mirror image $x'$. By part~(a), the color of $x'$ in $H$ is determined (namely, it must be black), and the proof is complete.

$(ii)$. In $H$, each $a_i$ is incident to precisely one vertex $v\notin\ell$. Since by $(i)$ the color of $v$ is determined, so is the color of $a_i$. Since, by part (a), $v$ and its mirror image $v'$ have opposite colors, the color of $a_i$ is reversed if the cut at $a_i$ is reversed. The same argument  works for the~$b_i$'s.
\end{proof}

\parindent0pt
{\it Proof of Theorem $\ref{tgb}$.} Without loss of generality we may assume that $G$ is connected (otherwise just apply the statement for each of its connected components, and multiply).

\parindent12pt
Set $k=\w(G)$. We prove first the special case when  $\{a_1,\dotsc,b_k\}$ is an independent set, and each vertex in it has precisely one incident edge from above, and one from below.

We claim that if $H$ is a doubly reduced subgraph of $G$, $x$ is the number of its vertices $a_i$ that have the same color as $a_1$, and $y$ is the number of its vertices $b_i$ that have color opposite to the color of $a_1$, then $H$ is balanced\footnote{ I.e., has the same number of vertices in its two bipartition classes.} if and only if $x=y$.

Indeed, suppose for definiteness that $a_1$ is white, and let $\al$ and $\be$ be the number of white and black vertices of $H$ above $\ell$, respectively. By Lemma \ref{tgd}(a), the number of white vertices of $H$ below $\ell$ is $\be$, and the number of its black vertices below $\ell$ is $\al$. Thus, the total number of white vertices of $H$ is $\al+\be+x+(k-y)$, while the total number of its black vertices is $\be+\al+(k-x)+y$. These are equal precisely if $x=y$, as claimed.

By Lemma \ref{tgd}(b)$(ii)$, there are exactly two doubly reduced subgraphs of $G$ in which all $a_i$'s have the same color, and all $b_i$'s have the opposite color. Out of these two, let $H_1$ be the one in which the cut above $a_1$ was chosen.

Let $G_1,\dotsc,G_{2^k}$ be the reduced subgraphs of $G$ (recall that reduced subgraphs are obtained from $G$ by cutting only at the $a_i$'s), $G_1$ being the one in which the cuts at the $a_i$'s are the same as in $H_1$.

We clearly have
\begin{align}
\M(G)=\sum_{i=1}^{2^k}\M(G_i).
\label{egb}
\end{align}
Furthermore, by Lemma \ref{tga}, the subgraphs $G_1,\dotsc,G_{2^k}$ have the same weighted count of perfect matchings. Therefore \eqref{egb} implies
\begin{align}
\M(G)=2^k\M(G_1).
\label{egc}
\end{align}
Let $H_1,\dotsc,H_{2^k}$ be the subgraphs of $G_1$ obtained by cutting above or below $b_i$, $i=1,\dotsc,k$, in all possible ways (the doubly reduced subgraph $H_1$ defined above is clearly one of them). Then 
\begin{align}
\M(G_1)=\sum_{i=1}^{2^k}\M(H_i).
\label{egd}
\end{align}
However, by the third paragraph in this proof, in order for $H_i$ to be balanced, it needs to have the same number of $a_i$'s of the same color as $a_1$ as $b_i$'s of opposite color to $a_1$. Our definitions of $G_1$ and $H_1$, together with Lemma \ref{tgd}(b)$(ii)$, imply that all $a_i$'s have the same color in each $H_i$. It follows that all $b_i$'s must have the opposite color in order for $H_i$ to be balanced. Lemma \ref{tgd}(b)$(ii)$ implies that this happens in precisely one $H_i$, namely in $H_1$. As all the other summands in \eqref{egc} are zero\footnote{ Since a bipartite graph has no perfect matching unless it is balanced.}, equation \eqref{ega} follows from \eqref{egc}, \eqref{egd}, and the readily checked fact (which follows from Lemma \ref{tgd}) that the cuts made in $H_1$ are precisely the ones made in the definition of $G'$ (which implies that $G'=H_1$ in this case)

The general case follows using the vertex splitting lemma \cite[Lemma 1.3]{FT}, as in the proof of Theorem 1.2 in \cite{FT} .
\hfill $\square$

\section{Square cylinder graphs of odd girth}

As pointed out in the previous section, for odd $m$ the approach presented there for finding $\M(C_{m,n})$ does not work. However, there is an alternative approach that works for odd $m$ (but not when $m$ is even!), based on the TFK formula and the factorization theorem of \cite{FT}. The result we obtain this way for odd $m$ can be put together with the formula we obtained in Theorem \ref{tea} for even $m$, obtaining the following.

\begin{theo}
\label{tfa}
For all non-negative integers $m$ and $n$ we have\footnote{ For even $m$ this becomes precisely Theorem \ref{tea}. Suppose therefore that $m$ is odd. Then the factors in the first three products are always non-negative, and the ones in the fourth are always less or equal than zero. Therefore for even $n$, the overall product under the fourth root is non-negative. If $n$ is odd, then --- as we are assuming $m$ odd --- $C_{m,n}$ has an odd number of vertices, and thus no perfect matchings. This agrees with the given formula, as in that case the factor in the fourth product corresponding to $k=(n+1)/2$ is zero.} 
\begin{align}
\M(C_{m,n})&=
\nonumber
\\[10pt]
&\!\!\!\!\!\!\!\!\!\!\!\!\!\!\!\!\!\!\!\!\!\!\!\!\!
\sqrt[4]{
\prod_{k=1}^n
\left[1+\left(\!\cos\frac{k\pi}{n+1}+\sqrt{1+\cos^2\frac{k\pi}{n+1}}\,\right)^{\!\!m}\,\right]
\!\left[1+\left(\!\cos\frac{k\pi}{n+1}-\sqrt{1+\cos^2\frac{k\pi}{n+1}}\,\right)^{\!\!m}\,
\right]
}
\nonumber
\\[10pt]
&\!\!\!\!\!\!\!\!\!\!\!\!\!\!\!\!\!\!\!\!\!\!\!\!\!\!\!\!\!\!\!\!
\times
\sqrt[4]{
\prod_{k=1}^n
\left[1+\left(\!-\cos\frac{k\pi}{n+1}+\sqrt{1+\cos^2\frac{k\pi}{n+1}}\,\right)^{\!\!m}\,\right]
\!\left[1+\left(\!-\cos\frac{k\pi}{n+1}-\sqrt{1+\cos^2\frac{k\pi}{n+1}}\,\right)^{\!\!m}\,
\right]
}.
\label{efa}
\end{align}

\end{theo}


In fact, the alternative approach mentioned above leads to a formula (see \eqref{tfb} below) that looks quite different from \eqref{tfa}. We prove this formula first, and then show that it agrees with~\eqref{tfa}.

\begin{theo} 
\label{tfb}
For non-negative integers $m$ and $n$ we have
\begin{equation}
\M(C_{2m+1,2n})=\prod_{k=1}^n\prod_{j\in\{1,3,\dotsc,2m+1\}}\left(4\cos^2\frac{j\pi}{4m+2}+4\cos^2\frac{k\pi}{2n+1}\right).
\label{efb}
\end{equation}

\end{theo}

\begin{figure}[t]
\vskip0.2in
\centerline{
\hfill
{\includegraphics[width=0.45\textwidth]{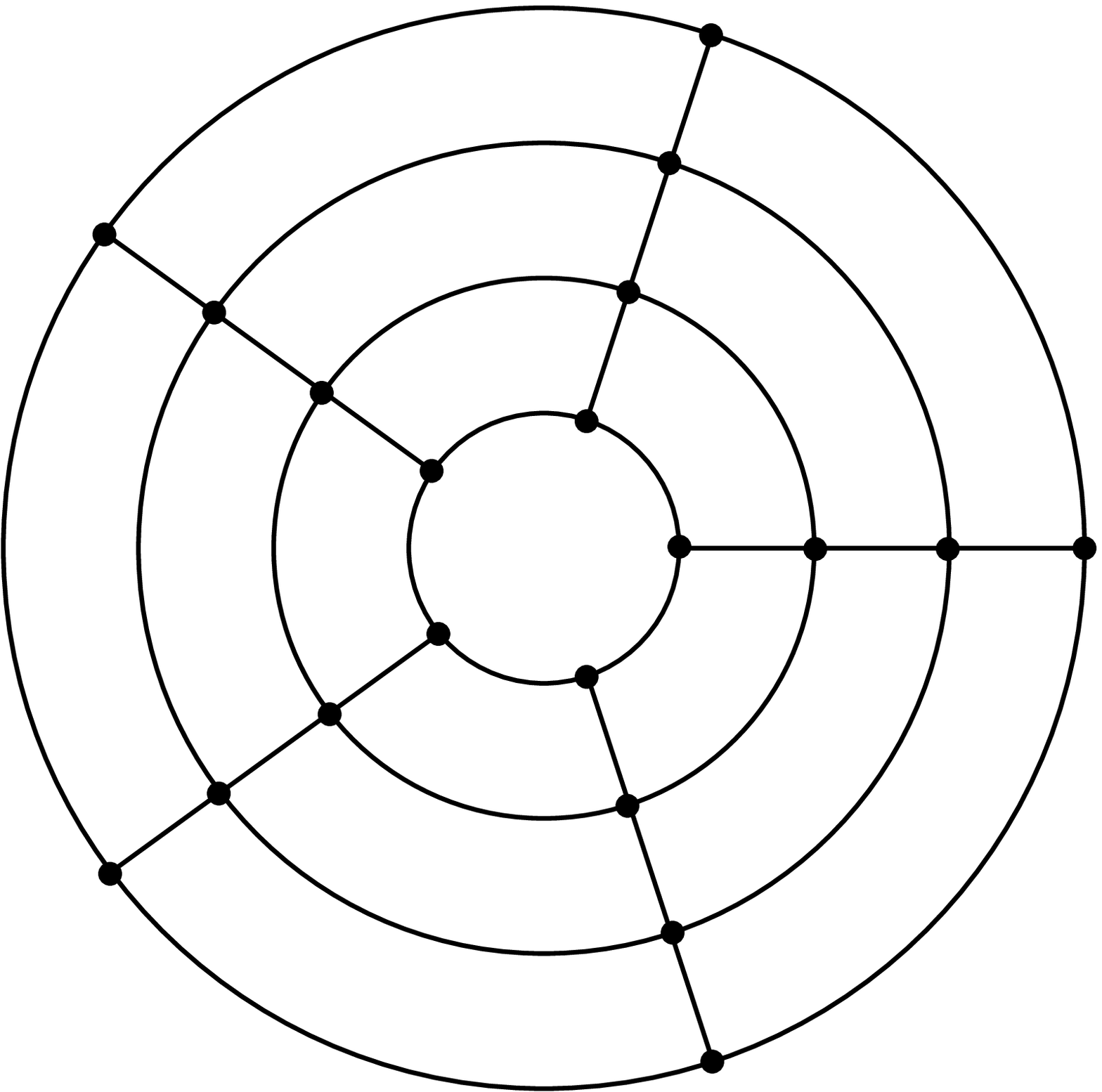}}
\hfill
{\includegraphics[width=0.175\textwidth]{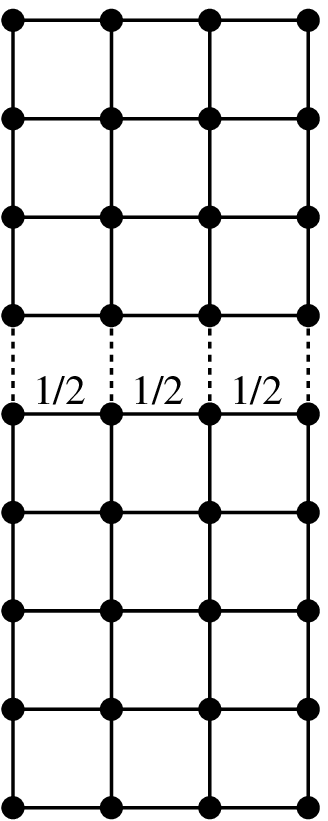}}
\hfill
}
\caption{\label{ffa} Embedding the rectangular cylinder grid $C_{5,4}$ in the plane (left). Applying the factorization theorem to the rectangular grid $R_{9,4}$ (right).}
\end{figure}

\begin{proof} Embed the rectangular cylinder graph $C_{2m+1,2n}$ in an annulus as indicated on the left in Figure \ref{ffa}. Then $C_{2m+1,2n}$ is clearly symmetric and turn-bipartite. Apply Theorem \ref{tgb} for $G=C_{2m+1,2n}$. One readily sees that the resulting graph $G'$ is the $(2m+1)\times2n$ rectangular grid graph $\dot{R}_{2m+1,2n}$, with the $2n-1$ edges along the top weighted by $1/2$, and all others weighted by 1. Then Theorem \ref{tgb} gives
\begin{equation}
\M(C_{2m+1,2n})=2^n \M(\dot{R}_{2m+1,2n}).
\label{efc}  
\end{equation}
In turn, $\M(\dot{R}_{2m+1,2n})$ can be expressed in terms of numbers of matchings of unweighted rectangular grid graphs as follows. Consider the rectangular grid graph $R_{4m+1,2n}$, and apply to it the original factorization theorem \cite[Theorem 1.2]{FT} (this is illustrated in the picture on the right in Figure \ref{ffa}). This yields
\begin{equation}
\M(R_{4m+1,2n})=2^n \M(R_{2m,2n}) \M(\dot{R}_{2m+1,2n}).
\label{efd}  
\end{equation}
Combining \eqref{efc} and \eqref{efd} leads to
\begin{equation}
\M(C_{2m+1,2n})=\frac{\M(R_{4m+1,2n})}{\M({R}_{2m+1,2n})}.
\label{efe}  
\end{equation}
Using the TFK formula (see \cite{FT,Kast})
\begin{equation}
\M(R_{m,n})=\prod_{j=1}^{\lceil m/2\rceil} \prod_{k=1}^{\lceil n/2\rceil}
\left(4\cos^2\frac{j\pi}{m+1}+4\cos^2\frac{k\pi}{n+1}\right)
\label{eff}
\end{equation}
in \eqref{efe}, all factors at the denominator cancel out, and one obtains formula \eqref{efb}.
\end{proof}

\begin{lem} For $z$ an indeterminate we have 
\label{tfc}
%
\begin{align}
&
\prod_{k\in\{1,3,\dotsc,2m+1\}}\left(4z^2+4\cos^2\frac{k\pi}{4m+2}\right)
=
\nonumber
\\[10pt]
&\ \ \ \ \ \ \ \ \ \ \ \ \ \ \ \ \ \ \ \ 
2z\sqrt{
-\left[1-\left(z+\sqrt{1+z^2}\right)^{4m+2}\right]
\left[1-\left(z-\sqrt{1+z^2}\right)^{4m+2}\right]
}.
\label{efg}
\end{align}
\end{lem}

\parindent0pt
{\it Proof.}
The $n$th Chebyshev polynomial of the first kind, $T_n(x)$, is defined by $T_n(\cos\theta)=\cos(n\theta)$, and thus satisfies the recurrence $T_{n+1}(x)=2xT_n(x)-T_{n-1}(x)$. Together with the initial values~$T_0(x)=1$ and $T_1(x)=x$, this implies that $T_n(x)$ is a polynomial of degree $n$, and that for $n\geq1$ its leading coefficient is $2^{n-1}$. By its defining equation, the roots of $T_n(x)$ are then $\cos\frac{(2k+1)\pi}{2n}$, $k=0,\dotsc,n-1$. Thus we have
\begin{equation}
T_{2m+1}(x)=\frac{2^{2m}}{x}\prod_{k\in\{1,3,\dotsc,2m+1\}}\left(x-\cos\frac{k\pi}{4m+2}\right)
\left(x+\cos\frac{k\pi}{4m+2}\right),
\label{efh}
\end{equation}
where the factor $x$ in the denominator accounts for the fact that the root $x=0$ appears twice in the product above. Using this, we obtain
\begin{align}
\prod_{k\in\{1,3,\dotsc,2m+1\}}\left(4z^2+4\cos^2\frac{k\pi}{4m+2}\right)
&=
\prod_{k\in\{1,3,\dotsc,2m+1\}}\left(2z-2i\cos\frac{k\pi}{4m+2}\right)
\left(2z+2i\cos\frac{k\pi}{4m+2}\right)
\nonumber
\\[10pt]
&=(-2i)^{2m+2}\,\frac{iz}{2^{2m}}\,T_{2m+1}(iz)=(-1)^{m+1}\,4iz\,T_{2m+1}(iz).
\label{efi}
\end{align}
On the other hand, we have 
\begin{align}
\left[1-\left(z+\sqrt{1+z^2}\right)^{4m+2}\right]
&\left[1-\left(z-\sqrt{1+z^2}\right)^{4m+2}\right]
\nonumber
\\[10pt]
&=
1-\left(z+\sqrt{1+z^2}\right)^{4m+2}-\left(z-\sqrt{1+z^2}\right)^{4m+2}+1
\nonumber
\\[10pt]
&=-\left[\left(z+\sqrt{1+z^2}\right)^{2m+1}+\left(z-\sqrt{1+z^2}\right)^{2m+1}\right]^2,
\label{efj}
\end{align}
so the right hand side of \eqref{efg} becomes
\begin{equation}
2z\left[\left(z+\sqrt{1+z^2}\right)^{2m+1}+\left(z-\sqrt{1+z^2}\right)^{2m+1}\right].
\label{efk}
\end{equation}
Therefore, to complete the proof it suffices to show that the right hand side of \eqref{efi} agrees with \eqref{efk}, which in turn amounts to
\begin{equation}
T_{2m+1}(iz)=\frac12\frac{(z+\sqrt{1+z^2})^{2m+1}+(z-\sqrt{1+z^2})^{2m+1}}{(-1)^{m+1} i}.
\label{efl}
\end{equation}
Writing $(-1)^{m+1} i=(-i)^{2m+1}$, this readily follows from the well-known fact (see e.g. \cite{AS}) that the Chebyshev polynomial of the first kind $T_n(x)$ can be written as
\begin{equation}
T_n(x)=\frac{(x+\sqrt{x^2-1})^n+(x-\sqrt{x^2-1})^n}{2}.\hskip 1in\square
\label{efm}
\end{equation}
%

\parindent0pt
{\it Proof of Theorem $\ref{tfa}$.}
For even $m$, the statement follows directly from Theorem \ref{tea}. Suppose therefore that $m$ is odd. Then, if $n$ is also odd, the expression on the right hand side of \eqref{efa} is zero, and \eqref{efa} holds, as $C_{m,n}$ has then an odd number of vertices, and thus no perfect matchings.

\parindent12pt
For the remaining case, write the cylinder graph as $C_{2m+1,2n}$, to spell out the parities of its parameters. One readily sees that \eqref{efa} is then equivalent to
\begin{align}
\M(C_{2m+1,2n})&=
\nonumber
\\[10pt]
&\!\!\!\!\!\!\!\!\!\!\!\!\!\!\!\!\!\!\!\!\!\!\!\!\!\!\!\!\!\!\!\!\!\!\!\!\!\!\!\!\!
\sqrt{(-1)^n
\prod_{k=1}^n
\left[1+\left(\!\cos\frac{k\pi}{2n+1}+\sqrt{1+\cos^2\frac{k\pi}{2n+1}}\,\right)^{\!\!2m+1}\,\right]
\!\left[1+\left(\!\cos\frac{k\pi}{2n+1}-\sqrt{1+\cos^2\frac{k\pi}{2n+1}}\,\right)^{\!\!2m+1}\,
\right]
}
\nonumber
\\[10pt]
&\!\!\!\!\!\!\!\!\!\!\!\!\!\!\!\!\!\!\!\!\!\!\!\!\!\!\!\!\!\!\!\!\!\!\!\!\!\!\!\!\!\!\!\!\!\!\!\!
\times
\sqrt{
\prod_{k=1}^n
\left[1+\left(\!-\cos\frac{k\pi}{2n+1}+\sqrt{1+\cos^2\frac{k\pi}{2n+1}}\,\right)^{\!\!2m+1}\,\right]
\!\left[1+\left(\!-\cos\frac{k\pi}{2n+1}-\sqrt{1+\cos^2\frac{k\pi}{2n+1}}\,\right)^{\!\!2m+1}\,
\right]
}.
\label{efn}
\end{align}
Apply Lemma \ref{tfc} to each inner product in equation \eqref{efb}, once for each $z=\cos\frac{k\pi}{2n+1}$, $k=1,\dotsc,n$. Since\footnote{ One way to see this is to note that the Chebyshev polynomial of the second kind $U_{2n}(x)$ has leading coefficient $2^{2n}$, constant term $(-1)^n$, and zeros $\cos\frac{k\pi}{2n+1}$, $k=1,\dotsc,2n$.} $\prod_{k=1}^n\cos\frac{k\pi}{2n+1}=\frac{1}{2^n}$, this yields equation \eqref{efn}.
\hfill $\square$

\medskip
\parindent0pt
{\it Remark $7$.} Formula \eqref{efa} looks quite different from formula \eqref{efb}, but as we saw above the two are the same when $m$ is odd in the former. In the proof of Lemma \ref{tfc}, the fact that the $m$-parameter (denoted there by $2m+1$) is odd was essential in the proof (both for equations~\eqref{efj} and~\eqref{efl}). It doesn't seem that for even $m$ one can rewrite formula \eqref{efa} in the style of the TFK formula.


\parindent12pt



\section{Concluding remarks}

In this paper we presented a new method for solving the two dimensional dimer problem. Our general result (see Theorem \ref{tba}) expresses the number of dimer coverings of graphs built from a certain type of linear building blocks (called strands) as the determinant of a matrix which is the product of the bi-adjacency matrices of its strands, or the latter plus a multiple of the identity matrix. For the case of the hexagonal and square lattices, we evaluated the resulting determinants and obtained explicit product formulas for the number of dimer coverings of hexagonal and square cylinder graphs (see Theorems \ref{tca} and \ref{tfa}). An interesting (and puzzling) feature of our formulas is that they look quite different from the classical TFK formula (see equation \eqref{eff}), which enumerates the dimer coverings of the closely related rectangular grid graphs. For one thing, our formulas have a linear number of factors, while the number of factors in the TFK formula is quadratic. A partial connection between these two types of formulas was given in Section 7. It would be interesting to elucidate further the relationship between them.


\end{document}